\documentclass[12pt]{amsart}

\usepackage[a4paper,centering,margin=2.5cm]{geometry}
\usepackage{cite}
\usepackage{color,hyperref,cleveref}
\definecolor{darkblue}{rgb}{0.0,0.0,0.3}
\hypersetup{colorlinks,breaklinks,
            linkcolor=darkblue,urlcolor=darkblue,
            anchorcolor=darkblue,citecolor=darkblue}

\usepackage{amssymb}
\usepackage{graphicx,psfrag}
\def\boxit{$\sqcap\kern-8pt\sqcup$}

\newcommand{\R}{\mathbb{R}}
\newcommand{\Ort}{\mathrm{ort}}

\newcommand{\p}{\mathbb{P}}

\theoremstyle{plain}
\newtheorem{theorem}{Theorem}[section]
\newtheorem{lemma}[theorem]{Lemma}
\newtheorem{proposition}[theorem]{Proposition}
\newtheorem{corollary}[theorem]{Corollary}

\newtheorem{remark}[theorem]{Remark}
\newtheorem{question}[theorem]{Question}

\theoremstyle{definition}
\newtheorem{definition}[theorem]{Definition}
\newtheorem{conjecture}[theorem]{Conjecture}
\newtheorem*{conjecture*}{Conjecture}
\newtheorem{example}[theorem]{Example}

\newtheorem*{problem*}{Problem}

\newcommand{\M}{\mathcal{M}}

\title{On $k$-neighborly reorientations of oriented matroids}

\dedicatory{}

\author{Rangel Hern\'andez-Ortiz}
\address{Universitat Rovira i Virgili, Departament d'Enginyeria Inform\`{a}tica i Matem\`{a}tiques, Av. Pa\"{i}sos Catalans 26, 43007 Tarragona, Spain}
\email{rangel.hernandez@urv.cat}

\author{Kolja Knauer}
\address{Departament de Matem\`atiques i Inform\`atica, Universitat de Barcelona, Spain, \newline LIS, Aix-Marseille Universit\'e, CNRS, and Universit\'e de Toulon, Marseille, France.}
\email{kolja.knauer@ub.edu}

\author{Luis Pedro Montejano}
\address{Serra H\'unter Fellow, Universitat Rovira i Virgili, Departament d'Enginyeria Inform\`{a}tica i Matem\`{a}tiques, Av. Pa\"{i}sos Catalans 26, 43007 Tarragona, Spain}
\email{luispedro.montejano@urv.cat}


\date{\today}

\begin{document}

\begin{abstract}
We study the existence and the number of $k$-neighborly reorientations of an oriented matroid. This leads to $k$-variants of McMullen's problem and Roudneff's conjecture, the case $k=1$ being the original statements. 
 Adding to results of Larman and García-Colín, we provide new bounds on $k$-McMullen's problem and prove the conjecture for several ranks and $k$ by computer. Further, we show that $k$-Roudneff's conjecture for fixed rank and $k$ reduces to a finite case analysis. As a consequence we prove the conjecture for odd rank $r$ and $k=\frac{r-1}{2}$ as well as for rank $6$ and $k=2$ with the aid of the computer.
\end{abstract}

\maketitle

\section{Introduction}\label{intro}
We assume some knowledge and standard notation of the theory of oriented matroids, for further reference the reader can consult the standard reference  \cite{BVSWZ99}.
An \emph{oriented matroid} is a pair $\M=(E,\mathcal{C})$ of a finite \emph{ground set} $E$ and a set of \emph{sign-vectors} $\mathcal{C}\subseteq \{+,-,0\}^E$ called \emph{circuits} satisfying a certain set of axioms, see \Cref{def:circuits}. The \emph{size} of a member  $X\in\{+,-,0\}^E$ is the size of its \emph{support} $\underline{X}=\{e\in E\mid X_e\neq 0\}$. Throughout the paper all oriented matroids are considered \emph{simple}, i.e., all circuits have size at least $3$. The \emph{rank} $r$ of $\M$ is the size of a largest set not containing a circuit of $\M$.  An oriented matroid of rank $r$ is called \emph{uniform} if all its circuits are of size $r+1$. Most of the problems we study in this paper reduce to uniform oriented matroids.

\smallskip

For a \emph{sign-vector} $X\in\{+,-,0\}^E$ on ground set $E$ we denote by $X^+:=\{e\in E\mid X_e=+\}$, the set of \emph{positive elements} of $X$, and $X^-:=\{e\in E\mid X_e=-\}$, its set of
\emph{negative elements}. Hence, the set $\underline{X}=X^+\cup X^-$ is the {support} of $X$. For a subset $R\subseteq E$ the \emph{reorientation of $R$} is the oriented matroid $_{-R}\M$ obtained from $\M$ by reversing the sign of $X_e$ for every $e\in R$ and $X\in\mathcal{C}$. The set of all oriented matroids that can be obtained this way from $\M$ is the \emph{reorientation class} $[\M]$ of $\M$.
 We denote by $-X$ the sign-vector $_{-E}X$ where all signs are reversed, i.e., such that $-X^+=X^-$ and $-X^-=X^+$.
 We say that $X$ is \emph{positive} if $X^-=\emptyset$ and $\underline{X}\neq\emptyset$.

\begin{definition}\label{def:circuits}
An oriented matroid $\mathcal{M}=(E,\mathcal{C})$ is a pair of a finite ground set $E$ and a collection of signed sets on $E$ called \emph{circuits}, satisfying the following axioms:
\begin{itemize}
\item[(C0)] $\emptyset\notin\mathcal{C}$,
\item[(C1)] $X\in \mathcal{C}$ if and only if $-X\in \mathcal{C}$,
\item[(C2)] if $X,Y\in\mathcal{C}$ and $\underline{X}\subseteq \underline{Y}$, then $X=\pm Y$,
\item[(C3)] if $X,Y\in\mathcal{C}$, $X\neq- Y$, and $e\in X^+\cap Y^-$, then there exists a $Z\in \mathcal{C}$, such that $Z^+\subseteq (X^+\cap Y^+)\setminus e$ and $Z^-\subseteq (X^-\cap Y^-)\setminus e$.
\end{itemize}
\end{definition}

An oriented matroid $\mathcal{M}=(E,\mathcal{C})$ is \emph{acyclic} if every circuit has positive and negative signs, i.e., if $|X^+|>0$ and $|X^-|>0$ for every circuit $X\in \mathcal{C}$.
We say that $\mathcal{M}$ is    \emph{$k$-neighborly} if for every subset $R\subseteq E$ of size at most $k$ the reorientation $_{-R}\M$ is acyclic, in other words, if  $|X^+|>k$ and $|X^-|>k$, for every $X\in \mathcal{C}$. A reorientation $R$ of $\mathcal{M}$ is \emph{$k$-neighborly} if $_{-R}\M$ is $k$-neighborly.

\smallskip

In this paper we study $k$-neighborly reorientations of oriented matroids.
In particular we study two well-known conjectures on $1$-neighborly oriented matroids and their generalization to arbitrary $k$: the $k$-McMullen problem and the $k$-Roudneff conjecture.

\smallskip

Along the course of the paper we will provide and use different equivalent descriptions of being $k$-neighborly: in terms of faces (\Cref{lem:faces}), in terms of orthogonality (\Cref{equivalentforms}) in terms of topes (defined in Subsection \ref{section_orthogonality}) and in terms of balls in the tope graph (\Cref{prop:topegraph}).
It follows from the definition that a $k$-neighborly oriented matroid is  $k'$-neighborly for all $0\leq k'\leq k$.
Note that $\M$ is $0$-neighborly if and only if $\M$ is acyclic. If $\M$ is $1$-neighborly, then $\M$ is called \emph{matroid polytope}.
If $\M$ has rank $r$, then it can be at most $\lfloor\frac{r-1}{2}\rfloor$-neighborly and in this case $\M$ is often just called \emph{neighborly}. There is quite some work about {neighborly oriented matroids}, starting with Sturmfels~\cite{S88} and~\cite[Section 9.4]{BVSWZ99} but also more recent works such as~\cite{MP15,P13}. In the realizable setting (see the appendix for the definition of a realizable oriented matroid), a $k$-neighborly matroid is a $k$-neighborly polytope, i.e., a polytope such that every set of at most $k$ vertices are the vertices of a face.

\smallskip

Given an oriented matroid $\M=(E,\mathcal{C})$ the \emph{contraction} of $F\subset E$ is the oriented matroid $\mathcal{M}/F=(E\setminus F ,\mathcal{C}/F)$, where $\mathcal{C}/F$ is the set of support-minimal sign-vectors from $\{X\setminus F\mid X\in\mathcal{C}\}\setminus\{\mathbf{0}\}$ where $X\setminus F$ is the sign-vector on groundset $E\setminus F$ such that $(X\setminus F)_e=X_e$ for all $e\in E\setminus F$. If $\M$ is a uniform oriented matroid of rank $r$, then $\mathcal{M}/F$ is uniform of rank $\max\{0,r-|F|\}$.
The \emph{deletion} of $F$ from $\mathcal{M}$ is the oriented matroid $\mathcal{M}\setminus F=(E\setminus F ,\mathcal{C}\setminus F)$, where $\mathcal{C}\setminus F=\{X\setminus F\mid X\in\mathcal{C}, \underline{X}\cap F=\emptyset\}\setminus\{\mathbf{0}\}$. If $\M$ is a uniform oriented matroid of rank $r$, then $\mathcal{M}/F$ is uniform of rank $\min\{|E\setminus F|,r\}$.

\smallskip

Two sign-vectors $X,Y\in \{+,-,0\}^E$ are called \emph{orthogonal} if either there are $e,f\in E$ such that $X_eY_e=+$ and $X_fY_f=-$ or $\underline{X}\cap\underline{Y}{=\emptyset}$. The set $\mathcal{L}$ of \emph{covectors} of $\M$ consists of all sign-vectors $X\in \{+,-,0\}^E$ such that $X$ is orthogonal to every circuit $Y\in\mathcal{C}$ of $\M$. The set $\mathcal{C}^*$ of
\emph{cocircuits} of $\M$ consists of the support-minimal elements of $\mathcal{L}\setminus\{\mathbf{0}\}$. The \emph{topes} of $\M$ are defined as $\mathcal{L}\cap \{+,-\}^E$, i.e., as the maximal covectors of $\M$. All these three sets uniquely determine $\M$ and oriented matroids can be axiomatized in terms of them as well, see~\cite{BVSWZ99} for covectors, cocircuits, and~\cite{daS95,Han90} for topes. The \emph{tope graph} $\mathcal{G}(\mathcal{M})$ is the graph whose vertex set are the topes of $\M$, where two vertices are adjacent if they differ in a single coordinate. The (unlabelled) $\mathcal{G}(\mathcal{M})$ determines the reorientation class of $\mathcal{M}$ and purely graph theoretical polynomial time verifiable characterizations of tope graphs of oriented matroids have been obtained recently, see~\cite{KM20}.

\smallskip

The definition of the base axioms of an oriented matroid is due to Guiterrez Novoa~\cite{Gut65} who called the structure multiply ordered sets. Lawrence~\cite{L82} proved that this yields another axiom system for oriented matroids. The term \emph{chirotope} is due to Dress~\cite{Dre85}, who reinvented oriented matroids.

\begin{definition}
Let $r\ge 1$ be an integer  be a set. An oriented matroid of rank $r$ is a pair $\M=(E,\chi)$ of a finite ground set $E$ and a \emph{chirotope} $\chi:E^r\to \{+,0,-\}$ satisfying:
\begin{itemize}
    \item[(B0)] $\chi$ is not identically zero,
    \item[(B1)] $\chi$ is alternating,
    \item[(B2)] if $x_1,\ldots, x_r,y_1, \ldots, y_r\in E$ and $\chi(y_i,x_2,\ldots, x_r)\chi(y_1, \ldots, y_{i-1},x_1,y_{i+1},\ldots, y_r)\geq 0$ for all $i\in[r]$, then $\chi(x_1,\ldots, x_r)\chi(y_1, \ldots, y_r)\geq 0$.
\end{itemize}

\end{definition}

If $\chi: E^r \rightarrow \{-,+\}$ is a chirotope, then $\M=(E,\chi)$ is uniform.
Moreover, if $E=[n]$ and $\chi(B)=+$ for any ordered tuple $B = (b_1<\ldots <b_r)$, then the uniform matroid $\M=(E,\chi)$ is the alternating oriented matroid of rank $r$ on $n$ elements: ${C}_r(n)$. In this model the \emph{reorientation} $_{-R}\M$ is obtained by defining $$_{-R}\chi(x_1,\ldots, x_r)=(-\chi(x_1,\ldots, x_r))^{|\{x_1,\ldots, x_r)\}\cap R|}.$$

Given an oriented matroid $\M=(E,\chi)$ of rank $r$, its \emph{dual} is the oriented matroid $\M^*=(E,\chi^*)$ of rank $n-r$ defined by setting $$\chi^*(x_1,\ldots, x_{n-r})=\chi(x'_1,\ldots, x'_r)\mathrm{sign}(x_1,\ldots, x_{n-r},x'_1,\ldots, x'_r),$$ where $(x'_1,\ldots, x'_r)$ is any permutation of $E\setminus\{x_1,\ldots, x_{n-r}\}$. In particular, $\M$ is uniform of rank $r$ if and only if $\M^*$ is uniform of rank $n-r$. Therefore, we have the following observation which will be useful throughout this paper.

\begin{remark}\label{oneclass_r+2}
For every $r\leq n\le r+2$, there is exactly one reorientation class of uniform rank $r$ oriented matroids on $n$ elements.
\end{remark}

\subsection{$k$-McMullen's problem}\label{intro_k_McMull}
This problem is about the existence of $k$-neighborly reorientations of uniform oriented matroids.
Denote by $\nu(r,k)$ (respectively $\nu_{\mathrm{R}}(r,k)$) the largest $n$ such that any (realizable) uniform oriented matroid $\M$ of rank $r$ and $n$ elements has a $k$-neighborly reorientation. Clearly, $\nu(r,k)\leq \nu_{\mathrm{R}}(r,k)$ for $k\ge 1$.
Note that since every uniform oriented matroid has an acyclic orientation $\nu(r,0)=+\infty$.
 This parameter was originally only studied for $k=1$. First, it was defined for realizable uniform oriented matroids~\cite{L72} (the original version can be found in the appendix)
 and then for general uniform oriented matroids~\cite{CS85}. The following conjecture is known as the \emph{McMullen problem} \cite{L72}:
\begin{conjecture}[McMullen 1972]\label{conjecture_McMullen}
\emph{For any $r\ge 3$ it holds $\nu(r,1)=2r-1$.}
\end{conjecture}

The inequality $2r-1\le \nu_{\mathrm{R}}(r,1)$ has been shown in~\cite{L72} for
realizable uniform oriented matroids and in~\cite{CS85} for general {uniform} oriented matroids, i.e., $2r-1\le \nu(r,1)$.
This conjecture has been verified for $r\leq 5$~\cite{FLS01}, but remains open otherwise. After a series of results~\cite{L72,L86}, the currently best-known upper bound $\nu (r,1)< 2(r-1)+\left\lceil {\frac {r}{2}}\right\rceil$ is due to Ramírez Alfonsín~\cite{R01}. For general positive $k$ we propose the following strengthening of \Cref{conjecture_McMullen}:

\begin{question}\label{question_k-McMullen}
 Does $\nu(r,k)= r+\lfloor\frac{r-1}{k}\rfloor$ hold for all $k=0,\ldots,\lfloor\frac{r-1}{2}\rfloor$ and $r\ge 3$?
\end{question}

\begin{remark}\label{bounds_nu_R}
In \cite[Theorem 1]{GL15} the authors stated that  $r+\lceil\frac{r-1}{k}\rceil\le \nu_{\mathrm{R}}(r,k)<2r-k-1$, for $2\le k\le \lfloor\frac{r-1}{2}\rfloor$, but what they actually prove was $$r+\lfloor\frac{r-1}{k}\rfloor\le \nu_{\mathrm{R}}(r,k)<2r-k+1.$$
\end{remark}
\begin{proof}
The authors stated this result in \cite{GL15} in terms of the dimension $d$, instead of the rank $r$, where $d=r-1$. So, we may reformulate the parameter $\nu_{\mathrm{R}}(r,k)$ in terms of $d$ as
$\nu'_{\mathrm{R}}(d,k)$.
To prove the lower bound, the authors introduce the parameter $\lambda(r-1,k)$ proving that $\lambda(r-1,k)\le (k+1)(r-1)+(k+2)$ \cite[Lemma 9]{GL15}. On the other hand, they proved  in \cite[Equation (1)]{GL15} that
 \begin{equation}\label{lambda_nu}
 \nu'_{\mathrm{R}}(r-1,k)=max\{w \in \mathbb{N} : w\ge  \lambda(w-r-1,k) \}
\end{equation}
Hence, $\lambda(w-r-1,k) \le (k+1)(w-r-1)+(k+2)$ and so, for the positive integer $w$ such that $(k+1)(w-r-1)+(k+2) \le w$, we obtain by (\ref{lambda_nu}) that $w\le\nu'_{\mathrm{R}}(r-1,k)$. Now, from the inequality  $(k+1)(w-r-1)+(k+2) \le w$ it follows that $w \le r+\frac{r-1}{k}$ and since $w$ is an integer, $w\le r+\lfloor \frac{r-1}{k} \rfloor$. Therefore, we conclude by (\ref{lambda_nu}) that $r+\lfloor\frac{r-1}{k}\rfloor\le \nu'_{\mathrm{R}}(r-1,k)=\nu_{\mathrm{R}}(r,k)$.

To prove the upper bound, the authors construct in Definitions $8$ and $9$ of \cite{GL15} (for the cases $k=2$ and $k\ge3$, respectively) examples of realizable uniform oriented matroids, called Lawrence oriented matroids, of rank $r$ and order $2r-k+1$ (and not $2r-k-1$ as mentioned in Theorem $1$ of \cite{GL15}) without $k$-neighborly reorientations. Then, $\nu_{\mathrm{R}}(r,k)<2r-k+1$.
 \end{proof}

In fact, we show in \Cref{small_values_kMcMullen} that the inequality $r+\lceil\frac{r-1}{k}\rceil\le \nu_{\mathrm{R}}(r,k)$ does not hold for $r=6$ and $k=2$. On the other hand, notice that  $r+\lfloor\frac{r-1}{k}\rfloor=2r-k-1$ for $r=5$ and $k=2$, showing that $\nu_{\mathrm{R}}(r,k)<2r-k-1$ is not true.

\smallskip

Question \ref{question_k-McMullen} holds in a sense for $k=0$ since every uniform oriented matroid has an acyclic orientation, i.e.,  $\nu(r,0)=+\infty$. For $k=1$, is just \Cref{conjecture_McMullen} (McMullen's problem) 
and for realizable uniform oriented  matroid it is partially solved since $r+\lfloor\frac{r-1}{k}\rfloor\le \nu_{\mathrm{R}}(r,k)$ (Remark \ref{bounds_nu_R}). Note that yet another variant of McMullen's problem has been studied recently in \cite{garcia2023number}.

\smallskip

 In \Cref{k-McMullen_results} we prove that $r-1+\lfloor\frac{r-1}{k}\rfloor \le \nu(r,k)$  for $r\ge 3$ and every $k=1,\ldots,\lfloor\frac{r-1}{2}\rfloor$ (\Cref{cotainf1v(rk)}). Hence, we are only off by $1$ of the lower bound claimed in \Cref{question_k-McMullen}. Moreover, we  prove the lower bound $r+\lfloor\frac{r-1}{k}\rfloor \le \nu(r,k)$ for some cases (\Cref{cotainf3_v(rk)}).
Further, with the help of the computer, we improve in Section \ref{computer} the upper bound $\nu_{\mathrm{R}}(r,k)<2r-k+1$ for some small values of $r$ and $k$ and as a consequence, we answer \Cref{question_k-McMullen} affirmatively showing that $\nu(5,2)=\nu_{\mathrm{R}}(5,2)=7$, $\nu(6,2)=\nu_{\mathrm{R}}(6,2)=8$
and $\nu(7,3)=\nu_{\mathrm{R}}(7,3)=9$ (\Cref{small_values_kMcMullen}).

\subsection{$k$-Roudneff's conjecture}\label{intro_k_Roudf}
This problem is about the number of $k$-neighborly reorientations of a rank $r$ oriented matroid  {$\mathcal{M}$} on $n$ elements, denoted by $m(\M,k)$.  The \emph{cyclic polytope} of dimension $d$ with $n$ vertices, $C_d(t_1,\ldots, t_n)$,
discovered by Carath\'eodory \cite{Car}, is
the  convex hull in $\R^d$ of $n\ge d+1\ge3$ different points
$x(t_1),\dots ,x(t_n)$ of the moment curve $\mu: \R\longrightarrow
\R^d, \ t \mapsto (t,t^2,\dots ,t^d)$.
Cyclic polytopes play an important role in the combinatorial  convex geometry due to their
connection with certain extremal problems. The Upper Bound theorem due to McMullen is an example of this \cite{Mac}. 
The oriented matroid associated to the cyclic polytope of dimension $d=r-1$ with $n$ elements is the \emph{alternating oriented matroid} ${C}_r(n)$. It is the uniform oriented matroid of rank $r$ and ground set $E=[n]:=\{1, \ldots,n\}$ such that every circuit $X\in \mathcal{C}$ is \emph{alternating}, i.e., $X_{i_j}=-X_{i_{j+1}}$ for all $1\leq j\leq r$ if $\underline{X}=\{i_1, \ldots, i_{r+1}\}$ and $i_1<\ldots<i_{r+1}$.

\smallskip

Denote by $c_r(n,k)=m({C}_r(n),k)$ the number of $k$-neighborly reorientations of ${C}_r(n)$. Since ${C}_r(n)$ is uniform and the $0$-neighborly reorientations are just the acyclic reorientations, we have $c_r(n,0)=2\sum_{i=0}^{r-1}{{n-1}\choose i}$,  see e.g.~\cite{Cor80}. Roudneff \cite{R91} proved that $c_r(n,1)\ge {2}\sum_{i=0}^{r-3}
\binom{n-1}{i}$ and that is an equality for all
$n\ge 2r-1$. In \cite{FR01} it is shown that $c_r(n,1)={2}(\binom{r-1}{n-r+1}+\binom{r}{n-r}+\sum_{i=0}^{r-3} \binom{n-1}{i}{)}$ for $n\ge r+1$.

\smallskip

The following has been conjectured by Roudneff,  originally in terms of projective pseudohyperplane arrangements \cite[Conjecture 2.2]{R91} (the original version can be found in the appendix).


\begin{conjecture}[Roudneff 1991]\label{conj1}
For any rank $r\ge 3$ oriented matroid $\mathcal{M}$ on $n\ge 2r-1$ elements it holds $m(\M,1)\le c_r(n,1)$.
\end{conjecture}

The above conjecture is stated for $r\ge 3$ since for $r=1,2$ there is only one reorientation class and clearly $m(\M,1)= c_r(n,1)$ (see \cite[Section 6.1]{BVSWZ99}). The case $r=3$ is not difficult to prove, the case $r=4$  has been shown in \cite{R99} and recently also for $r=5$~\cite{HOKMS23} as well as for Lawrence oriented matroids~\cite{MR15}.
Furthermore, in~\cite{BBLP95} it is shown that for realizable oriented matroids of rank $r$ on $n$ elements, the number of $1$-neighborly reorientations is $2\binom{n}{r-3}+O(n^{r-4})$, i.e., Roudneff's conjecture holds asymptotically in the realizable setting.
In~\cite[Question 2]{MR15} the authors asked if the conjecture holds for $n\ge r+1$ and again, it turns out that it is true for $r\leq 5$~\cite{HOKMS23} and for  Lawrence oriented matroids~\cite{MR15}.
In the same manner as for McMullen's problem we propose the $k$-variant of the above question.

\begin{question}\label{quest:kRoudneff}
Is it true that  $m(\mathcal{M},k)\le c_r(n,k)$, for any rank $r\ge 3$ oriented matroid $\mathcal{M}$ on $n> r$ elements and $0\leq k\leq \lfloor\frac{r-1}{2}\rfloor$?
 \end{question}
The above question holds for $k=0$, since all oriented matroids of given rank $r$ and number $n$ of elements have at most the number of acyclic reorientations of (any) uniform oriented matroid~\cite{Cor80}, moreover it holds trivially for $n\le r+2$ (\Cref{oneclass_r+2}). For $k=1$ \Cref{quest:kRoudneff} combines Roudneff's conjecture and~\cite[Question 2]{MR15}. Hence, the answer is positive for $k=1$ if $r\leq 5$.
 We prove in \Cref{prop:onlythebaseishard} that in order to answer \Cref{quest:kRoudneff} affirmatively for a fixed $r$ and $k$, it is enough to prove it for uniform oriented matroids with $r+1\leq n\leq 2(r-k)+1$ and all rank $r'\le r$ uniform oriented matroid $\M'$ on $n'=2(r'-k)+1$ elements.
Thus, the question reduces to a finite number of cases. As a consequence we answer \Cref{quest:kRoudneff} in the positive  for odd $r$ and $k=\lfloor\frac{r-1}{2}\rfloor$  ({\Cref{r_odd_max_ort}}). Moreover, in \Cref{coro_cyclique_unique} we answer \Cref{quest:kRoudneff}  in the affirmative for $r=6$ and $k=2$. One might think that as in the Upper Bound Theorem~\cite{Mac} the reorientation class of $\mathcal{C}_r(n)$ is unique in attaining the maximum in \Cref{quest:kRoudneff}. However, in \Cref{thm_computer} we show that different reorientation classes attain $c_r(n,k)$, for $r=5,n=8,9$ and $k=2$ and for $r=7,n=10$ and $k=3$.

\smallskip

As a tool in the study of \Cref{quest:kRoudneff} we make use of a refinement of $m(\M,k)$, namely we define {the \emph{$o$-vector} of $\M$, as the vector with entries  $o(\M,k)$, for every $k=0,1,\ldots,\lfloor\frac{r-1}{2}\rfloor$, where $o(\M,k)$ is}
the number of reorientations of $\M$ that are $k$-neighborly but not $(k+1)$-neighborly. 
In {\Cref{O-formula_neighb}} we compute this parameter for the alternating oriented matroid, which lies at the heart of the proof of \Cref{prop:onlythebaseishard}. We then proceed to study $o(\M,k)$ as a parameter of independent interest, and note that here the role of the alternating oriented matroid is more complicated than for $m(\M,k)$.
On the one hand, the alternating oriented matroid maximizes $o(\M, k)$ for  $n\le r+2$ (\Cref{oneclass_r+2}), for odd $r$ and $k=\lfloor\frac{r-1}{2}\rfloor$  (\Cref{r_odd_max_ort}) and for $r=6$ and $k=2$ (\Cref{coro_cyclique_unique}). In \Cref{thm_computer} we show that  for $r=5$, $k=1$ and $n=8,9$, for $r=6$, $k=2$ and $n=9$,  and for $r=7$, $k=2$ and $n=10$, the alternating oriented matroid is even unique (up to reorientation) with this property. On the other hand, for $r=6$, $k=1$ and $n=9$ and for $r=7$, $k=1$ and $n=10$, there are (up to reorientation) $91$ and $312336$  uniform oriented matroids $\mathcal{M}$ of rank $r$ on $n$ elements such that $o(\mathcal{M},1)> o(\mathcal{C}_6(9),1)$ and $o(\mathcal{M},1)> o(\mathcal{C}_7(10),1)$, respectively (\Cref{thm_computer} (c) and (d)).

\smallskip

\subsection{Organization of the paper}
The structure of the paper is as follows.
\smallskip


In Section \ref{intro}, we introduce some basic notions of oriented matroid theory and explain the $k$-McMullen problem and the $k$-Roudneff conjecture.

\smallskip

In Section \ref{k-McMullen_results} we study $k$-McMullen's problem providing some lower bounds of  $\nu(r,k)$.

\smallskip

In Section~\ref{sec:Roudneff} we study $k$-Roudneff's conjecture. 
First, we present in Subsection \ref{section_orthogonality} two cryptomorphic descriptions of $k$-neighborliness (Propositions \ref{equivalentforms} and
\ref{prop:topegraph}), most importantly the notion of  $k'$-orthogonality which generalizes usual orthogonality of sign-vectors, as well as metric descriptions in terms of the tope graph.
In Subsection \ref{section_cyclic} we study the tope graph of the alternating oriented matroid and obtain  $o(C_r(n),k)$ for $n$ large enough (\Cref{O-formula_neighb}).
In Subsection \ref{section_general} we show that we may restrict \Cref{quest:kRoudneff} to uniform oriented matroids and a finite case analysis (\Cref{prop:onlythebaseishard}) and solve  \Cref{quest:kRoudneff} for odd $r$ and $k=\frac{r-1}{2}$ (\Cref{r_odd_max_ort}).

\smallskip

 In Section \ref{computer} we present computational results for the $k$-McMullen problem and the $k$-Roudneff conjecture. In Subsection \ref{computer_program} we explain our computer program that obtains $o(\M,k)$ for uniform oriented matroids $\M$ via the chirotope. 
We then obtain the maximal $o(\M,k)$ among all $\M$ of rank $r$ and $n$ elements for several values of $r$, $n$, and $k$ (Theorem \ref{thm_computer}). On the one hand, we answer \Cref{question_k-McMullen} affirmatively for  $(r,k)\in (5,2),(6,2),(7,3)$ and show that the lower bound in \Cref{question_k-McMullen} is tight in one more case, i.e., $10\le \nu(7,2)$ (Theorem \ref{small_values_kMcMullen}). On the other hand, using \Cref{prop:onlythebaseishard} we answer
\Cref{quest:kRoudneff}  in the affirmative for  $r=6,k=2$ and $n\ge 9$ (\Cref{coro_cyclique_unique}).

\smallskip

Finally, in the appendix we  present McMullen's problem and Roudneff's conjecture in their original versions.


\section{Results on the $k$-McMullen problem}\label{k-McMullen_results}

First, we present the following description of  $k$-neighborly oriented matroids that will be useful in this section.
Following Las Vergnas~\cite{LV80} a subset $F\subseteq E$ is a \emph{face} of $\mathcal{M}=(E,\mathcal{L})$
 if there is a covector $Y\in\mathcal{L}$ such that $F=E\setminus Y^+$ and $Y^-=\emptyset$.

\begin{lemma}\label{lem:faces}
An oriented matroid $\M=(E,\mathcal{C})$ is $k$-neighborly if and only if every  subset $F\subseteq E$ of size at most $k$ is a {face}.
\end{lemma}
\begin{proof}
Suppose that $\M=(E,\mathcal{C})$ is $k$-neighborly and let $F\subseteq E$ be of size $k$. Then for every subset $F'\subseteq F$ the reorientation $_{-F'}\M$ is acyclic. Thus, we have that $F'$ properly intersects $X^+$ or $X^-$ for every circuit $X$ of $\mathcal{M}$ or is disjoint if $X^+, X^-\neq \emptyset$. This implies that the sign-vector $Y$ that is positive on $F'\setminus E$ and $0$ on $F'$ is orthogonal to all circuits of $\M$. Hence all subsets of $F$ are faces.

Conversely, if for every $F\subseteq E$ of size $k$ all subsets $F'\subseteq F$ are faces, then $_{-F}\M$ is acyclic. This implies that $\M=(E,\mathcal{C})$ is $k$-neighborly.
\end{proof}

The inequality $2r-1\le \nu(r,1)$ has been shown in~\cite{CS85}.
In this section we give some lower bounds for $\nu(r,k)$ and $k\ge 2$.
In order to prove that $n\le\nu(r,k)$, one has to prove that any uniform rank $r$ oriented matroid $\M$ on $n$ elements has a reorientation that is $k$-neighborly. 
In particular, if $n\le r+2$ then $\M$ is in the same reorientation class as the alternating oriented matroid by \Cref{oneclass_r+2}. So, we have the following observation.

\begin{remark}\label{McMullenr+2}
  If $\lfloor\frac{r-1}{k}\rfloor\le2$, then $r+\lfloor\frac{r-1}{k}\rfloor \le \nu(r,k).$
\end{remark}

\begin{theorem}\label{cotainf1v(rk)}
 For every  $k=2,\ldots,\lfloor\frac{r-1}{2}\rfloor$, we have $r-1+\lfloor\frac{r-1}{k}\rfloor\le \nu(r,k)$.
\end{theorem}
\begin{proof}
Let  $\mathcal{M}=(E, \mathcal{C}^*)$ be a uniform rank $r$ oriented matroid on $n=r-1+\lfloor\frac{r-1}{k}\rfloor$ elements with set of cocircuits $\mathcal{C}^*$. First notice that  $|\underline{C}|=\lfloor\frac{r-1}{k}\rfloor$ for every cocircuit $C\in \mathcal{C}^*$ and since $n\geq (k+1)\lfloor \dfrac{r-1}{k}\rfloor$, there exist $k+1$ cocircuits $C_{1},\ldots, C_{k+1}\in \mathcal{C}^*$ mutually disjoint. Let $R$ be the set of elements $x\in\bigcup_{i=1}^{k+1} \underline{C_{i}}$ such that $x\in \underline{C_{i}^{-}}$ for some $i\in \{1,\ldots, k+1\}$ and consider $_{-R}\mathcal{M}$, the oriented matroid resulting from reorienting the set $R$. We will see that $_{-R}\mathcal{M}$ is a $k$-neighborly oriented matroid. Let $S\subseteq E$ be any set  of size at most $k$ and observe that $S\cap \underline{C_i}=\emptyset$ for some $i\in\{1,\ldots,k+1\}$. Then, $S\subseteq E\setminus\underline{C_i}$, where $\underline{C_i}$ is the support of a positive cocircuit in $_{-R}\mathcal{M}$, concluding  that $S$ is a face of $_{-R}\mathcal{M}$. Therefore, $_{-R}\mathcal{M}$ is a $k$-neighborly  by Lemma \ref{lem:faces} and the theorem holds.
\end{proof}



The next result improves \Cref{cotainf1v(rk)} in  some cases.

\begin{theorem}\label{cotainf3_v(rk)}
 For every  $k=2,\ldots,\lfloor\frac{r-1}{2}\rfloor$, we have  $r+\lfloor\frac{r-1}{k}\rfloor \le \nu(r,k)$
 if  $r-1\equiv \beta \bmod k$, where $\beta=\lceil \frac{k-1}{2}\rceil,\ldots,k-1$,
\end{theorem}
\begin{proof}
Let $\mathcal{M}=(E,\mathcal{C^{*}})$ be a uniform rank $r$ oriented matroid on $n=|E|=r+\lfloor\frac{r-1}{k}\rfloor$ elements with set of cocircuits $\mathcal{C}^{*}$ and notice that  $|\underline{C}|=\lfloor\frac{r-1}{k}\rfloor+1$ for every cocircuit $C\in \mathcal{C}^*$. As $r-1=\alpha k + \beta$ for some positive integer $\alpha$, then  $\alpha=\lfloor\frac{r-1}{k}\rfloor$ and so, $n=(k+1)\alpha+\beta+1$.
\medskip

Next, we will consider a partition of $E$ into two sets, $A$ and $B$, as follows. Consider any set   $B= \{b_1,\ldots,b_{\beta+1}\}$ of cardinality $\beta+1$ and let $A=\bigcup_{i=1}^{k+1}A_i=E\setminus B$, where $|A_i|=\alpha$ and $A_i\cap A_j=\emptyset$ for every distinct $i,j\in \{1,\ldots, k+1\}$. Notice that $|B|\geq\lceil \frac{k+1}{2}\rceil$ since by hypothesis $\beta\geq \lceil \frac{k-1}{2}\rceil$. Thus, for every $i=1,\ldots,\lfloor\frac{k+1}{2}\rfloor$ we will consider the sets $$D_i=A_{2i-1}\cup A_{2i}\cup b_i$$ of cardinality $2\alpha +1$ and if $k$ is even, we will also consider the set $D_{\frac{k+2}{2}}= A_{k+1}\cup b_{\frac{k+2}{2}}$,  of cardinality $\alpha+1$ (see Figure \ref{construcionDi}).

\begin{figure}[htb]
\begin{center}
 \includegraphics[width=.5 \textwidth]{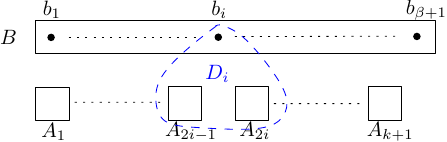}
\label{construcionDi}
\end{center}
\end{figure}

First observe  that $D_i\cap D_j = \emptyset$  for every distinct $i,j\in \{1,\ldots,\lceil\frac{k+1}{2}\rceil \}$.
Now, for each $i\in \{1,\ldots,\lfloor\frac{k+1}{2}\rfloor\}$ consider $\mathcal{M}/(E\setminus D_i)$, the contraction of $E\setminus D_i$ from $\mathcal{M}$, which we will denote by $\mathcal{M}_{D_i}$. 
So, $\mathcal{M}_{D_i}$ is a uniform oriented matroid with ground set $D_i$ and rank $r_i=\max\{0,r-|E\setminus D_i|\}$ (see Section \ref{intro} for the definition of a contraction of a uniform oriented matroid). As $n=(k+1)\alpha+\beta+1$ and $r=\alpha k +\beta+1$, it follows that $r_i=r-|E\setminus D_i|=r-n+2\alpha +1=\alpha+1$. Hence, the cocircuits of $\mathcal{M}_{D_i}$ have
cardinality $|D_i|-r_i+1=|D_i|-(\alpha+1)+1=\alpha+1$ and so, the cocircuits of  $\mathcal{M}_{D_i}$ are also cocircuits of $\mathcal{M}$.  On the other hand, $\mathcal{M}_{D_i}$ has order $2\alpha+1=2r_i-1$
and since $2r_i-1\le \nu(r_i,1)$ (see \cite{CS85}), there exists a reorientation of $\mathcal{M}_{D_i}$ on $R_i\subseteq D_i$ such that $_{-R_i}\mathcal{M}_{D_i}$
is a $1$-neighborly matroid polytope.

 Now, consider the set
\begin{equation*}
    R = \left\{
	       \begin{array}{ll}
		 \bigcup\limits_{i=1}^{\lfloor\frac{k+1}{2}\rfloor}R_i    & if\  k\ is\ odd,  \\
		 \bigcup\limits_{i=1}^{\lfloor\frac{k+1}{2}\rfloor}R_i\cup \underline{F}^-   & if\  k\ is\ even, \\
	       \end{array}
	     \right.
\end{equation*}
where $F$  denotes the cocircuit of $\mathcal{M}$ with support $\underline{F}=D_{\frac{k+2}{2}}$. Notice that $R_i\cap R_j=\emptyset$ for every distinct $i,j=1,\ldots,\lfloor\frac{k+1}{2}\rfloor$ and moreover,  if $k$ is even then also
   $R_i\cap \underline{F}^-=\emptyset$  for every  $i=1,\ldots,\lfloor\frac{k+1}{2}\rfloor$. We will see that $_{-R}\mathcal{M}$ is a $k$-neighborly oriented matroid.
Let $S\subset E$ be any set  of size at most $k$ and first suppose that $|S\cap D_i|\leq 1$ for some $i \in \{1,\ldots,\lfloor\frac{k+1}{2}\rfloor\}$. As
$_{-R_i}\mathcal{M}_{D_i}$ is $1$-neighborly, then there exists a positive cocircuit $C$ of $_{-R_i}\mathcal{M}_{D_i}$ such that $\underline{C}\cap S=\emptyset$. Then,  $S$ is a face of $_{-R}\mathcal{M}$ since  $C$ is also a positive cocircuit of  $_{-R}\mathcal{M}$, concluding that $_{-R}\mathcal{M}$ is  $k$-neighborly  by Lemma \ref{lem:faces}. Now suppose that $|S\cap D_i|\ge 2$ for every $i=1,\ldots,\lfloor\frac{k+1}{2}\rfloor$ and let denote $D=\cup_{i=1}^{\lfloor\frac{k+1}{2}\rfloor}D_i$. Then, $|S\cap D|\ge 2\lfloor\frac{k+1}{2}\rfloor$ concluding that  $k$ is even since $|S|=k$. Moreover, as $|S\cap D|=k$, then $|S\cap \underline{F}| = 0$, concluding that $S$ is a face of $_{-R}\mathcal{M}$ since $\underline{F}$ is the support of a positive cocircuit of $_{-R}\mathcal{M}$. Therefore, $_{-R}\mathcal{M}$ is  $k$-neighborly  by Lemma \ref{lem:faces} and the theorem holds.
\end{proof}

\section{Results on the  $k$-Roudneff conjecture}\label{sec:Roudneff}

\subsection{Orthogonality and neighborliness}\label{section_orthogonality}
The main result of this subsection is to present an equivalent description of neighborliness in terms of orthogonality. Throughout this section and Section \ref{computer}, several of our results and proofs are stated in terms of  $k'$-orthogonality which is equivalent to $k$-neighborliness, with $k=k'-1$.

Given two sign-vectors $X,Y\in \{+,-,0\}^E$,  their \emph{separation} is the set $S(X,Y)=\{e\in E\mid X_e\cdot Y_e=-\}$. For convenience, we denote by $H(X,Y)=\{e\in E\mid X_e\cdot Y_e=+\}$. We define the \emph{orthogonality} of $X$ and $Y$ by $$X\perp Y=\min\{|H(X,Y)|, |S(X,Y)|\}.$$
We say that $X,Y$ are \emph{$k$-orthogonal} if $X\perp Y\geq k$. 
For an oriented matroid  $\M=(E,\mathcal{C})$, we call a sign-vector  $T\in \{+,-\}^E$, \emph{$k$-orthogonal} (to $\mathcal{M}$) if $X\perp T\geq k$ for all $X\in \mathcal{C}$. Note that for $X\in\{0,+,-\}^E$ and $T\in\{+,-\}^E$ of $\mathcal{M}$ we have $X\perp T\geq 1$ if and only if $X$ and $T$ are orthogonal. Hence, the sign-vectors $T\in \{+,-\}^E$ that are $1$-orthogonal to $\mathcal{M}$ constitute the set of topes $\mathcal{T}$ of $\mathcal{M}$. 
The following establishes a correspondence between $k-1$-neighborly reorientations and $k$-orthogonal topes.
\begin{proposition}\label{equivalentforms}
 Let $\mathcal{M}$ be a rank $r$ oriented matroid on $n=|E|$ elements
 , $R\subseteq E$, and $T$ the sign-vector  that is negative on $R$ and positive on $E\setminus R$, and $k={1},\ldots, \lfloor\frac{r+1}{2}\rfloor$. Then, $T$ is a $k$-orthogonal tope if and only if $_{-R}\M$ is $k-1$-neighborly.
\end{proposition}
\begin{proof}
"$\Rightarrow$". Let $T$ be a $k$-orthogonal tope of $\mathcal{M}$, i.e.,  $H(X,T)\ge k$ and $S(X,T)\ge k$ for every circuit $X$ of $\mathcal{M}$ and denote $R={T}^-$. Now, let us consider the oriented matroid $_{-R}\mathcal{M}$. Then, we notice that $T'$, the resulting sign-vector obtained from reorienting each element of $R$ in $T$, is a sign-vector with only $+$ entries, i.e., $T'=\{+\}^E$. Let $Y$ be the resulting circuit of $_{-R}\mathcal{M}$ obtaining from reorienting a circuit $X$ of $\mathcal{M}$.
Hence,  notice that $H(X,T)=H(Y,T')=|Y^+|$ and $S(X,T)=S(Y,T')=|Y^-|$, obtaining that $|Y^+|>k-1$ and $|Y^-|>k-1$ for every circuit $Y$ of $_{-R}\mathcal{M}$. Hence, $_{-R}\mathcal{M}$ is $(k-1)$-neighborly.

\smallskip

"$\Leftarrow$". Let $R\subseteq E$ {be} such that $_{-R}\M$ is $k-1$-neighborly and $T\in\{+,-\}^E$ be such that $T^-=R$. By definition $|Y^+|>k-1$ and $|Y^-|>k-1$  for every circuit  $Y$ of $_{-R}\mathcal{M}$. Let $T'$ be the tope of  $_{-R}\mathcal{M}$ such that $T'\in\{+\}^E$, which exists because $_{-R}\mathcal{M}$ is acyclic such that
$T$ is obtained by reoriented each element of $R$ in $T'$.  For every circuit  $Y$ of $_{-R}\mathcal{M}$, let $X$ be the resulting circuit of $\mathcal{M}$ obtaining from reorienting each element of $R'$ in $Y$.
Hence,   notice that $H(Y,T')=|Y^+|=H(X,T)$ and $S(Y,T')=|Y^-|=S(X,T)$, obtaining that $H(X,T)\ge k$ and $S(X,T)\ge k$ for every circuit $X$ of $\mathcal{M}$. Therefore $T$ is a $k$-orthogonal tope of $\mathcal{M}$.
\end{proof}


We say that $T$ is a {\em $k$-neighborly tope} of $\mathcal{M}$ if the  oriented matroid $_{-T^-}\M$ (obtained from reorienting such that $T$ becomes all positive) is $k$-neighborly. By \Cref{equivalentforms}, $k$-neighborly topes are in correspondence with the  $(k+1)$-orthogonal topes. Given an oriented matroid $\M$ with tope set $\mathcal{T}$ and $T\in \mathcal{T}$ denote
$\mathcal{M}^{\perp_k}=\{T\in\mathcal{T} \text{ } | \text{ } T \text{ is } k\text{-orthogonal}\}$. 
Further denote
$\Ort(T)=\max\{k \text{ } | \text{ } T \text{ is } k\text{-orthogonal}\}$
and  $\mathcal{O}_k(\mathcal{M})=\{T\in\mathcal{T} \text{ } | \text{ } \Ort(T)=k\}$, i.e.,  $\mathcal{O}_k(\mathcal{M})=\mathcal{M}^{\perp_k}\setminus \mathcal{M}^{\perp_{k+1}}$, the set of $k$-orthogonal but not $k+1$-orthogonal topes in $\mathcal{M}$.  So, we may describe the $o$-vector of $\mathcal{M}$ (defined in Subsection \ref{intro_k_Roudf}) in terms of orthogonality. By \Cref{equivalentforms} we get the following result that will be very useful throughout this paper.
\begin{corollary}\label{relation_ort_neighb}
  For every $k=1,\ldots, \lfloor\frac{r+1}{2}\rfloor$, $m(\M,k-1)=|\mathcal{M}^{\perp_k}|$. Moreover, $o(\M,k-1)=|\mathcal{O}_{k}(\M)|$.
\end{corollary}
Recall that the tope graph $\mathcal{G}(\M)$ of an oriented matroid $\M$ on $n$ elements naturally embeds as an induced subgraphs into the hypercube $\{+,-\}^E\cong Q_n$. The edges of $\mathcal{G}(\mathcal{M})$ are partitioned into classes, where two edges are equivalent if their corresponding adjacent topes differ in exactly the same entry $e\in E$ (see \Cref{figureB(3_5)B(3_4)} for an example and~\cite{KM20} for further reference).
In order to present a graph theoretical description of $k$-orthogonal topes for any graph $G$ with vertex $v$ and $k$ a non-negative integer, define  the \emph{ball} of radius $k$ and center $v$ in $G$, denoted by $B_k^G(v)$, as the induced subgraph of $G$ with set of vertices $V(B_k^G(v))=\{u\in V(G)\mid d_G(u,v)\leq k\}$. Here $d_G(u,v)$ denotes the distance between $u$ and $v$ in the graph $G$. Next, we have another characterization of neighborliness that we will use later.

\begin{proposition}\label{prop:topegraph}
    Let $\M$ be a rank $r$ oriented matroid on $n$ elements, $k={0},\ldots, \lfloor\frac{r-1}{2}\rfloor$, and $T\in \mathcal{T}$ a tope. Then $T$ is $k$-neighborly if and only if
 $B^{\mathcal{G}(\M)}_{k}(T)\cong B^{Q_n}_{k}(T)$.
 \end{proposition}
\begin{proof}
$T$ is $k$-neighborly if and only if reorienting any $F\subseteq E$ of size at most $k$ in $_{-T^-}\M$ is acyclic. Direct computation yields that this is equivalent to any reorientation of $_{-F}T$ being at least $1$-orthogonal. Hence this changing any set of at most $k$ coordinates in $T$ yields another vertex of $\mathcal{G}(\M)$. Since  $\mathcal{G}(\M)$ is an induced subgraph of $Q_n$ this is equivalent to $B^{\mathcal{G}(\M)}_{k}(T)\cong B^{Q_n}_{k}(T)$.
\end{proof}



\subsection{The $o$-vector of $\mathcal{C}_r(n)$ and its tope graph}\label{section_cyclic}
The main result of this subsection is to obtain the $o$-vector of
$\mathcal{C}_r(n)$, for $n\ge 2(r-k)+1$ (\Cref{O-formula_neighb}).

\smallskip

Let $n$ be a positive integer. Given a sign-vector $T\in \{+,-\}^n$, notice that $T$ partitions  $\{1,\ldots,n\}$ into signed  \emph{blocks}  $B_1,\ldots, B_m$, 
 $1\le m\le n$, where $B_j$ denotes the $j$-th maximal set of consecutive elements 
 having the same sign. Sign-vectors  with $m$ blocks can be easily counted as there are $m-1$ possibilities to choose the change of sign in $n-1$ places. So, the number of sign-vectors with $m$ blocks is 2$\binom{n-1}{m-1}$ since for every sign-vector $T$ there exists also $-T$.

\begin{remark}\label{numberBlocks} The number of sign-vectors of $n$ elements with $m$ blocks is $2\binom{n-1}{m-1}$.
\end{remark}

Throughout this subsection, we will denote by $\mathcal{T}$ and  $\mathcal{C}$ the set of topes and circuits of $\mathcal{C}_r(n)$, respectively. Given a tope  $T\in \mathcal{T}$  with $m$ blocks, we denote $O(m)=\lceil\frac{r+1-m}{2}\rceil$. Next, we prove that $O(m)$ is the minimum orthogonality that $T$ can have.

\begin{lemma}\label{minim_ort} Let $T\in \mathcal{T}$ be with $m$ blocks, then $\Ort(T)\ge O(m)$. 
\end{lemma}
\begin{proof}
Let $B_1,B_2,\ldots,B_m$
be the blocks of $T$ and let  $X\in \mathcal{C}$. As $X$  is alternating, we have that $\min\{|S(X,T)\cap B_i|,|H(X,T)\cap B_i|\}\ge\lceil\frac{|\underline{X}\cap B_i|-1}{2}\rceil$ for every $i=1,2,\ldots,m$.
Thus, $X\perp T=\min\{|H(X,T)|, |S(X,T)|\}\ge\sum\limits_{i=1}^{m}\lceil\frac{|X\cap B_i|-1}{2}\rceil\ge\lceil\frac{r+1-m}{2}\rceil=O(m)$ for every $X\in\mathcal{C}$,
concluding that $T$ is $O(m)$-orthogonal. Therefore, as $\Ort(T)=\max\{k \text{ }|\text{ } T \text{ is } k-\text{orthogonal}\}$, it follows that $\Ort(T)\ge O(m)$, concluding the proof.
\end{proof}

\begin{lemma}\label{B_rn=cyclicpoly}
Let $n\ge r+1\ge2$,
then $T\in \mathcal{T}$ if and only if $T$ has at most $r$ blocks. 
\end{lemma}
\begin{proof}
Recall that $T\in \mathcal{T}$ if and only if $T$ is $1$-orthogonal. On the other hand, as every circuit $X\in \mathcal{C}$ is alternating, it follows that   $T$ is $1$-orthogonal if and only if  $T$ has at most $r$ blocks. 
\end{proof}
For  $n\ge r+1\ge3$ we define the \emph{graph of blocks}, denoted by $B(r,n)$,  as the graph whose vertices are the sign-vectors $T\in \{+,-\}^n$ with at most $r$ blocks and two sign-vectors $T,T'$ are adjacent if $|S(T,T')|=1$ (see \Cref{figureB(3_5)B(3_4)}).
Thus, the adjacency of the vertices in the graph $B(r,n)$ corresponds to the adjacency of topes in $\mathcal{G}(\mathcal{C}_r(n))$ and so, it follows from \Cref{B_rn=cyclicpoly} that $B(r,n)$ is just the tope graph of $\mathcal{C}_r(n)$.
\smallskip

We say that a block is an \emph{even block} (\emph{odd block}) if it has an even (odd) number of elements. From now, we will denote  by $B_1,B_2,\ldots,B_m$  the blocks of a tope $T\in \mathcal{T}$. Denote by $\mathcal{B}_e$  and $\mathcal{B}_o$  the set of even and odd blocks of $T$, respectively. We will also denote $\mathbf{B}_e =|\mathcal{B}_e|$ and $\mathbf{B}_o=|\mathcal{B}_o|$.
Given a tope $T$ and a circuit $X$, let denote $S^X=\{i \text{ } | \text{ } \min\{|S(X,T)\cap B_i|, |H(X,T)\cap B_i|\}=|S(X,T)\cap B_i|\}$ and $H^X=\{i \text{ } | \text{ } \min\{|S(X,T)\cap B_i|, |H(X,T)\cap B_i|\}=|H(X,T)\cap B_i|\},$
see \Cref{remarkfigure}.

\newpage
\begin{figure}[htb]
\begin{center}
 \includegraphics[width=1\textwidth]{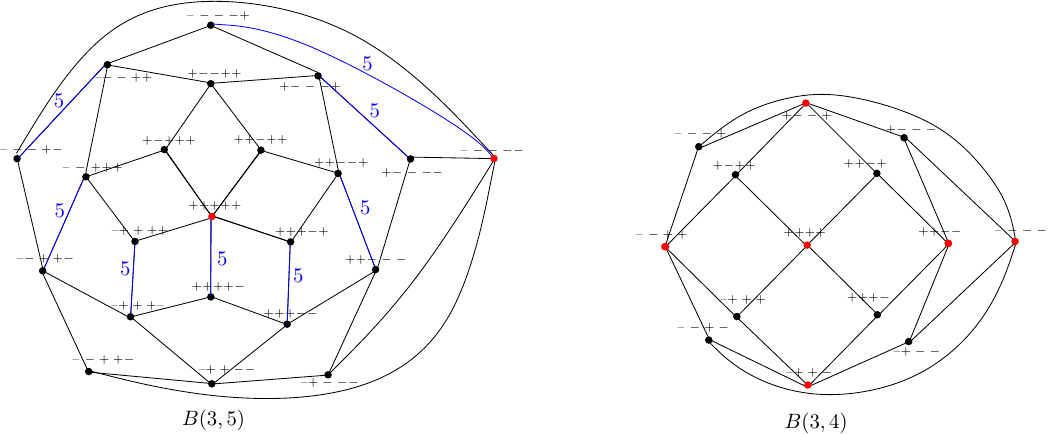}
 \caption{The graphs $B(3,5)$ and $B(3,4)$ corresponding to the tope graphs of $\mathcal{C}_3(5)$ and $\mathcal{C}_3(4)$, respectively. Red vertices correspond to $2$-orthogonal topes (i.e., $1$-neighborly topes) and adjacent topes between a blue edge, 
 differ exactly in the fifth entry.} \label{figureB(3_5)B(3_4)}
 \end{center}
\end{figure}
The following observation can be deduced by a simple
parity argument and will be very useful for the next lemmas.
\begin{remark}\label{blocks_even-odd}
Let  $T\in \mathcal{T}$  and $X\in \mathcal{C}$. Let $j<j'$ and suppose that $|\underline{X}\cap B_j|$ and $|\underline{X}\cap B_{j'}|$ are odd, and that $|\underline{X}\cap B_i|$ is even for every $i\in\{j+1,\ldots,j'-1\}$. Then, $$|\{j+1,\ldots,j'-1\}|=j'-j-1 \text{ is even  if and only if } j,j'\in S^X \text{ or } j,j'\in H^X.$$
\end{remark} 

The above remark holds since the blocks and the circuits are alternating. In the first example of \Cref{remarkfigure} we observe that if $j'=j+2$, then $j'-j-1=1$ is odd and so,  $j\in S^X\setminus H^X$ and $j'\in H^X\setminus S^X$ by \Cref{blocks_even-odd}. In the second example, we notice that $j'-j-1=0$ is even if $j'=j+1$
and so, $j,j'\in S^{X}$. Moreover, applying \Cref{blocks_even-odd} now to the blocks $j+1$ and $j+2$ of the second example, it follows that $j+1,j+2\in S^{X}$.
 \begin{figure}[htb]
\begin{center}
\includegraphics[width=0.9\textwidth]{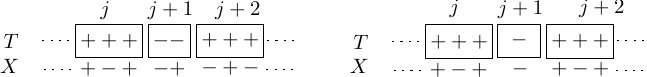}
 \caption{In the first example $j\in S^X\setminus H^X$, $j+1\in S^X\cap H^X$ and $j+2\in H^X\setminus S^X$. In the second example, $j,j+1,j+2\in S^{X}\setminus H^X$.} \label{remarkfigure}
 \end{center}
\end{figure}

The following lemmas provide a connection between blocks and orthogonality that will allow us to prove \Cref{O-formula_neighb}.
\begin{lemma} \label{lemmageneral:blocks-ortogonality}
Let $T\in \mathcal{T}$ be with $m$ blocks. If $r+1\le n-\mathbf{B}_e$, then
 $\Ort(T)=O(m).$
\end{lemma}
\begin{proof}
We will find a circuit  $X\in \mathcal{C}$ such that $X\perp T= O(m)$, where $O(m)=\lceil\frac{r+1-m}{2}\rceil$. First suppose that $r+1-m$ is even.
Then, there exists $X\in \mathcal{C}$ such that $|\underline{X}\cap B_i|$ is odd for every $1\le i\le m$, since $|\underline{X}|>m$ (\Cref{B_rn=cyclicpoly}) and $|\underline{X}|\le n-\mathbf{B}_e$.  
 Thus, by  \Cref{blocks_even-odd} and by possibly exchanging $X$ with $-X$, we can assume that $i\in S^X$ for every  $1\le i\le m$ and hence,
$$
X\perp T
\begin{array}[t]{l}
= \displaystyle \min\{|H(X,T)|, |S(X,T)|\}
=  \sum\limits_{i=1}^{m}|S(X,T)\cap B_i| \\ \ \\
= \displaystyle \sum\limits_{i=1}^{m}\frac{|\underline{X}\cap B_i|-1}{2}
=  \frac{(r+1-m)}{2}=O(m).
\end{array}
$$

Now suppose that $r+1-m$ is odd and consider $X\in \mathcal{C}$ such that $|\underline{X}\cap B_i|$ is odd for every $i\le m-1$ and $|(\bigcup\limits_{i=1}^{m-1} B_i)\cap \underline{X}|$ is maximum possible. 
 If there are no more elements in $\underline{X}$ to place, i.e., $|\underline{X}\cap B_i|$ is odd for every $1\le i\le m-1$ and $\underline{X}\cap B_m=\emptyset$, we obtain by \Cref{blocks_even-odd} and by possibly exchanging $X$ with $-X$, that $i\in S^X$ for every  $1\le i\le m-1$. Therefore,
$$
X\perp T
\begin{array}[t]{l}
= \displaystyle   \sum\limits_{i=1}^{m-1}|S(X,T)\cap B_i|
=\displaystyle \sum\limits_{i=1}^{m-1}\frac{|\underline{X}\cap B_i|-1}{2} \\ \ \\
= \displaystyle \frac{r+1-(m-1)}{2}=\frac{r+2-m}{2}=O(m),
\end{array}
$$
as desired. If  there are still elements in $\underline{X}$ to be placed, by the maximality of $|(\bigcup\limits_{i=1}^{m-1} B_i)\cap \underline{X}|$ and since $|\underline{X}|\le n-\mathbf{B}_e$,
we may place such elements in $B_m$ (see \Cref{figure_small_r+1}). Notice that $|\underline{X}\cap B_m|$ is even since $r+1-m$ is odd, obtaining that $m\in S^X\cap H^X$. Hence,   $i\in S^X$ for every  $1\le i\le m$ and then
$$X\perp T
\begin{array}[t]{l}
= \displaystyle \sum\limits_{i=1}^{m}|S(X,T)\cap B_i| =  \sum\limits_{i=1}^{m-1}\frac{|\underline{X}\cap B_i|-1}{2}+\frac{|\underline{X}\cap B_m|}{2} \\ \ \\
= \displaystyle  \frac{r+1-(m-1)}{2}=\frac{r+2-m}{2}=O(m).
\end{array}
$$
Therefore, there exists a circuit  $X$ such that $X\perp T= O(m)$ and so, $\Ort(T)\le O(m)$. Thus,   $\Ort(T)=O(m)$ by  \Cref{minim_ort}, concluding the proof.
\end{proof}

The example of Figure \ref{figure_small_r+1}, for $n=13, r+1=9,m=4$ and $\mathbf{B}_e=1$, illustrates the case of the above lemma when $r+1-m$ is odd. We choose a circuit $X$ such that $|\underline{X}\cap B_i|$ is odd for every $i\le m-1=3$ and  maximum possible. In this case, we notice that there are still elements in $\underline{X}$ to be placed, so, we may place such elements in $B_m=B_4$. Finally, we observe that $X\perp T=\min\{|H(X,T)|, |S(X,T)|\}=\min\{6, 3\}=3=O(m)$, as required.
 \begin{figure}[htb]
\begin{center}
\includegraphics[width=.5\textwidth]{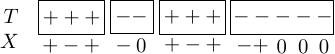}
 \caption{An example of the case $r+1-m$ odd in \Cref{lemmageneral:blocks-ortogonality}.} \label{figure_small_r+1}
 \end{center}
\end{figure}

Given a tope $T\in \mathcal{T}$ and a circuit $X\in \mathcal{C}$ of $\mathcal{C}_r(n)$, let us denote the sets $$S^X_{<j}=\Bigl\{ i \text{ } | \text{ } i<j \text{, } B_i\in \mathcal{B}_o  \text{ and } \min\{|S(X,T)\cap B_i|, |H(X,T)\cap B_i|\}=|S(X,T)\cap B_i|\Bigr\}$$ and $$H^X_{<j}=\Bigl\{i \text{ } | \text{ } i<j \text{, } B_i\in \mathcal{B}_o  \text{ and } \min\{|S(X,T)\cap B_i|, |H(X,T)\cap B_i|\}=|H(X,T)\cap B_i|\Bigr\},$$
see \Cref{figure_large_r+1_1,figure_large_r+1_2}. Notice that these sets  contain integers $i$ such that $B_i\in \mathcal{B}_o$ is an odd block.
The following lemma provides an upper bound for $\Ort(T)$ in terms of its blocks.
\begin{lemma} \label{lemmageneral:blocks-ortogonality2}
Let $T\in \mathcal{T}$ be with $m$ blocks and $n\ge r+2$.
 If $r+1> n-\mathbf{B}_e$, then
 $$\Ort(T)\le \frac{n-\mathbf{B}_e-m}{2}+r+1-(n-\mathbf{B}_e)+\left\lfloor\frac{\mathbf{B}_o}{2}\right\rfloor.$$ 
\end{lemma}
\begin{proof}
Let $\alpha=r+1-(n-\mathbf{B}_e)$.
In order to prove the lemma, it is enough to find $X\in \mathcal{C}$ such that
$X\perp T\le\frac{n-\mathbf{B}_e-m}{2}+\alpha+\lfloor\frac{\mathbf{B}_o}{2}\rfloor$. Let  us denote by $B^e_1,\ldots, B^e_{\alpha}$, the first $\alpha$ blocks of $\mathcal{B}_e$ (from left to right) and notice that $\alpha>0$ since by hypothesis $r+1> n-\mathbf{B}_e$. Then, $B^e_{\alpha}=B_j$ for some $j\in\{1,\ldots,m\}$. 
Consider  $X\in \mathcal{C}$ as follows:
\begin{itemize}
    \item $|\underline{X}\cap B_i|=|B_i|$ for every $B_i\in \mathcal{B}_o$,
    \item$|\underline{X}\cap B^e_i|=|B^e_i|$ for every $i=1,\ldots, \alpha$ and
    \item $|\underline{X}\cap B_i|=|B_i|-1$ otherwise.
\end{itemize}
Such a circuit exists, since
 $|\underline{X}|=n-\mathbf{B}_e+\alpha=r+1$. Thus,  as $|\underline{X}\cap B_i|$ is odd for every $i>j$, by  \Cref{blocks_even-odd} and by possibly exchanging $X$ with $-X$, we can assume that $i\in S^X$ for every  $i>j$.
Hence,
\begin{equation} \label{i>j}
\sum\limits_{i>j}|S(X,T)\cap B_i|
\begin{array}[t]{l}
= \displaystyle  \sum\limits_{i>j, B_i\in\mathcal{B}_e}|S(X,T)\cap B_i|+\sum\limits_{i>j, B_i\in \mathcal{B}_o}|S(X,T)\cap B_i|\\ \ \\
= \displaystyle \sum\limits_{i>j, B_i\in\mathcal{B}_e} \frac{|B_i|-2}{2} + \sum\limits_{i>j, B_i\in\mathcal{B}_o} \frac{|B_i|-1}{2}.\\ \ \\
\end{array}
\end{equation}
Consider the following cases.
\smallskip

\emph{Case $1$.
$|H^X_{<j}|\le \lfloor \frac{\mathbf{B}_o}{2}\rfloor$} (see  \Cref{figure_large_r+1_1}).

Then,
\begin{equation}\label{i<j}
\sum\limits_{i\le j}|S(X,T)\cap B_i|
\begin{array}[t]{l}
= \displaystyle  \sum\limits_{i\in S^X_{<j}\cup H^X_{<j}}|S(X,T)\cap B_i|+\sum\limits_{i\le j, B_i\in \mathcal{B}_e}|S(X,T)\cap B_i|\\ \ \\
= \displaystyle \sum\limits_{i\in S^X_{<j}} \frac{|B_i|-1}{2}+\sum\limits_{i\in H^X_{<j}} \frac{|B_i|+1}{2} + \sum\limits_{i\le j, B_i\in \mathcal{B}_e} \frac{|B_i|}{2}\\ \ \\
= \displaystyle \sum\limits_{i\in S^X_{<j}} \frac{|B_i|-1}{2}+\sum\limits_{i\in H^X_{<j}} \frac{|B_i|-1}{2} +|H^X_{<j}|+ \sum\limits_{i\le j, B_i\in \mathcal{B}_e} \frac{|B_i|-2}{2}+\alpha\\ \ \\
\le \displaystyle \sum\limits_{i\in S^X_{<j}\cup H^X_{<j}} \frac{|B_i|-1}{2} + \sum\limits_{i\le j, B_i\in \mathcal{B}_e} \frac{|B_i|-2}{2}+\alpha+\lfloor\frac{\mathbf{B}_o}{2}\rfloor.\\ \ \\
\end{array}
\end{equation}
Thus, by \Cref{i>j,i<j} we obtain that
$$
X\perp T
\begin{array}[t]{l}
=\min\{|H(X,T)|, |S(X,T)|\}
\le \displaystyle \sum\limits_{i\le j}|S(X,T)\cap B_i|+\sum\limits_{i>j}|S(X,T)\cap B_i|\\ \ \\
\le \displaystyle \sum\limits_{i\in S^X_{<j}\cup H^X_{<j}} \frac{|B_i|-1}{2} + \sum\limits_{i\le j, B_i\in \mathcal{B}_e} \frac{|B_i|-2}{2}+\alpha+\lfloor\frac{\mathbf{B}_o}{2}\rfloor\\ \ \\
+ \displaystyle \sum\limits_{i>j, B_i\in\mathcal{B}_e} \frac{|B_i|-2}{2} + \sum\limits_{i>j, B_i\in\mathcal{B}_o} \frac{|B_i|-1}{2}
=  \frac{n-\mathbf{B}_e-m}{2}+\alpha+\lfloor\frac{\mathbf{B}_o}{2}\rfloor\\ \ \\
\end{array}
$$
concluding that  $\Ort(T)\le \frac{n-\mathbf{B}_e-m}{2}+\alpha+\lfloor\frac{\mathbf{B}_o}{2}\rfloor$ and so, the lemma holds in this case.

\smallskip

\emph{Case $2$. $|H^X_{<j}|\ge \lceil \frac{\mathbf{B}_o}{2}\rceil$} (see  \Cref{figure_large_r+1_2}).
\smallskip

We will slightly  modify  $X$ in order to find another circuit $X'\in \mathcal{C}$ such that $X'\perp T\le\frac{n-\mathbf{B}_e-m}{2}+\alpha+\lfloor\frac{\mathbf{B}_o}{2}\rfloor$. First,
 observe that $\lceil\frac{\mathbf{B}_o}{2}\rceil\le|S^X_{<j}|+|H^X_{<j}|\le \mathbf{B}_o$. So, $|S^X_{<j}|+|H^X_{<j}|=\mathbf{B}_o-p$ for some $p\in\{0,\ldots,\lfloor\frac{\mathbf{B}_o}{2}\rfloor\}$. Then, $|S^X_{<j}|=\mathbf{B}_o-p-|H^X_{<j}|\le \mathbf{B}_o-p-\lceil\frac{\mathbf{B}_o}{2}\rceil$, concluding that
\begin{equation}\label{S^X_{<j}}
    |S^X_{<j}|\le \lfloor\frac{\mathbf{B}_o}{2}\rfloor-p
\end{equation}

 As $|S^X_{<j}|+|H^X_{<j}|+|\{i \text{ } | \text{ } i>j, B_i\in\mathcal{B}_o \}|=\mathbf{B}_o$, we also conclude that
\begin{equation}\label{B_o^{>j}=p}
    |\{i \text{ } | \text{ } i>j, B_i\in\mathcal{B}_o \}|=p
\end{equation}

Let $B_{j'}$ be the last block of $\mathcal{B}_e$ (from left to right). As $n\ge r+2$, we have that $\alpha<\mathbf{B}_e$ and then $j'>j$.  Choose $e\in B_j$ and $e'\in B_{j'}\setminus \underline{X}$ (such $e'$ exists since $|\underline{X}\cap B_{j'}|=|B_{j'}|-1$) and consider the circuit $\underline{X'}=\underline{X}-e\cup e'$. Hence, notice that
\begin{equation}\label{j and X'}
  S^X_{<j}=S_{<j}^{X'}
\end{equation}
Now, we claim that $j\in H^{X'}$. Let $i'=\max\{i \text{ } | \text{ } i<j \text{ and } B_i\in\mathcal{B}_o\}$ and first suppose that $i'\in S^X$. As we have assumed that $j+1\in S^X$, applying  \Cref{blocks_even-odd} to the blocks $B_{i'}$ and $B_{j+1}$, respect to the circuit $X$, we obtain that $(j+1)-i'-1$ is even. Hence, as $i'\in S^{X'}$ by \Cref{j and X'} and $j-i'-1$ is odd, applying   \Cref{blocks_even-odd} now to the blocks $B_{i'}$ and $B_{j}$, respect to the circuit $X'$, we obtain that $j\in H^{X'}$, as desired. The case when  $i'\in H^X$ can be treated analogously, so the claim holds. Thus,  by the above claim and by  \Cref{blocks_even-odd}  it can be deduced  that
\begin{equation}\label{j<i and X'}
    i\in H^{X'}  \text{ if } j\le i<j' \text{ and } i\in S^{X'} \text{ if } i>j' 
\end{equation}
Therefore, we conclude the following:
\vspace{0,5cm}

\begin{itemize}
    \item[(a)] $\sum\limits_{i\in S^{X'}_{<j}}|H(X',T)\cap B_i|=\sum\limits_{i\in S^{X'}_{<j}}\frac{|B_i|+1}{2}=\sum\limits_{i\in S^{X'}_{<j}}\frac{|B_i|-1}{2}+|S^{X'}_{<j}|\le \sum\limits_{i\in S^{X'}_{<j}}\frac{|B_i|-1}{2}+\lfloor\frac{\mathbf{B}_o}{2}\rfloor-p,$
      \vspace{0,3cm}

      by \Cref{S^X_{<j},j and X'}.
    \vspace{0,3cm}

      \item[(b)] $\sum\limits_{i\in H^{X'}_{<j}}|H(X',T)\cap B_i|=\sum\limits_{i\in H^{X'}_{<j}}\frac{|B_i|-1}{2}.$
       \vspace{0,3cm}

     \item[(c)] $|H(X',T)\cap B_{j'}| \text{ }+\sum\limits_{i<j \text{ and  } B_i\in \mathcal{B}_e}|H(X',T)\cap B_i|\begin{array}[t]{l}= \frac{|B_{j'}|}{2} \text{ } +\sum\limits_{i<j \text{ and  } B_i\in \mathcal{B}_e}\frac{|B_i|}{2} \\ \ \\ = \sum\limits_{i<j, i=j', \text{  } B_i\in \mathcal{B}_e}\frac{|B_i|-2}{2}+\alpha,\end{array}$
      \vspace{0,3cm}

     since such blocks are just $B^e_1,\ldots,B^e_{\alpha-1}$ and $B_{j'}$.

      \vspace{0,3cm}

     \item[(d)] $\sum\limits_{i\ge j, \text{ } i\neq j'} |H(X',T)\cap B_i|\begin{array}[t]{l} =\sum\limits_{j\le i<j', \text{ } B_i\in \mathcal{B}_e}\frac{|B_i|-2}{2} \text{ } +\sum\limits_{j<i<j', \text{ } B_i\in \mathcal{B}_o}\frac{|B_i|-1}{2} \text{ } +\sum\limits_{i>j'}\frac{|B_i|+1}{2}\\ \ \\ \le \sum\limits_{j\le i<j', \text{ } B_i\in \mathcal{B}_e}\frac{|B_i|-2}{2}+\sum\limits_{j<i<j', \text{ } B_i\in \mathcal{B}_o}\frac{|B_i|-1}{2}+\sum\limits_{i>j'}\frac{|B_i|-1}{2}+p,\end{array}$

by \Cref{B_o^{>j}=p,j<i and X'}.
\end{itemize}
\vspace{0,3cm}

Finally, by (a), (b), (c) and (d), we obtain that,

$$
X'\perp T
\begin{array}[t]{l}
= \sum\limits_{i< j}|H(X',T)\cap B_i|+\sum\limits_{i\ge j}|H(X',T)\cap B_i|\\ \ \\
\le \displaystyle \sum\limits_{i\in S^{X'}_{<j}\cup H^{X'}_{<j}} \frac{|B_i|-1}{2} +\lfloor\frac{\mathbf{B}_o}{2}\rfloor-p+\alpha+\sum\limits_{i< j, i=j', B_i\in \mathcal{B}_e} \frac{|B_i|-2}{2}\\ \ \\
+\displaystyle \sum\limits_{j\le i<j', B_i\in \mathcal{B}_e}\frac{|B_i|-2}{2}+\sum\limits_{j<i<j', B_i\in \mathcal{B}_o}\frac{|B_i|-1}{2}+\sum\limits_{i>j'}\frac{|B_i|-1}{2}+p\\ \ \\
=  \frac{n-\mathbf{B}_e-m}{2}+\alpha+\lfloor\frac{\mathbf{B}_o}{2}\rfloor,\\ \ \\
\end{array}
$$
concluding that  $\Ort(T)\le \frac{n-\mathbf{B}_e-m}{2}+\alpha+\lfloor\frac{\mathbf{B}_o}{2}\rfloor$ and so, the lemma holds.
\end{proof}
The example of \Cref{figure_large_r+1_1}, for $n=13, r+1=10,m=9,\mathbf{B}_e=4, \mathbf{B}_o=5, \alpha=1$ and $j=3$, illustrate Case $1$ of the above lemma. We consider the circuit $X$ as in the proof and we notice that $X\perp T=3=\frac{n-\mathbf{B}_e-m}{2}+\alpha+\lfloor\frac{\mathbf{B}_o}{2}\rfloor$, as required.
 \begin{figure}[htb]
\begin{center}
\includegraphics[width=.5\textwidth]{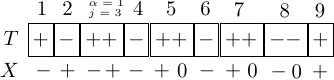}
 \caption{An example of the case $|H^X_{<j}|\le \lfloor \frac{\mathbf{B}_o}{2}\rfloor$ of  \Cref{lemmageneral:blocks-ortogonality2}. Notice that $2=|H^X_{<j}|=|H^X_{<3}|=\lfloor\frac{\mathbf{B}_o}{2}\rfloor$, $H^X_{<j}=\{1,2\}$,
 $S^X=\{3,4,5,6,7,8,9\}$ and $H^X=\{1,2,3\}$.} \label{figure_large_r+1_1}
 \end{center}
\end{figure}

The example of  \Cref{figure_large_r+1_2}, for $n=13, r+1=11,m=9,\mathbf{B}_e=4, \mathbf{B}_o=5, \alpha=2$ and $j=6$, illustrates the Case $2$ of \Cref{lemmageneral:blocks-ortogonality2}. We consider the circuit $X$ as in the proof and we notice that $X\perp T=5>4=\frac{n-\mathbf{B}_e-m}{2}+\alpha+\lfloor\frac{\mathbf{B}_o}{2}\rfloor$. Then, we consider the circuit $\underline{X'}=\underline{X}-e\cup e'$ and we observe that $X'\perp T=3<4=\frac{n-\mathbf{B}_e-m}{2}+\alpha+\lfloor\frac{\mathbf{B}_o}{2}\rfloor$, as required.
 \begin{figure}[htb]
\begin{center}
\includegraphics[width=.5\textwidth]{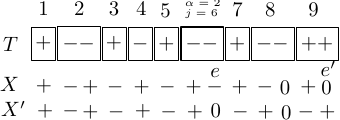}
 \caption{An example of the case $|H^X_{<j}|> \lfloor \frac{\mathbf{B}_o}{2}\rfloor$ of \Cref{lemmageneral:blocks-ortogonality2}. Notice that $|H^X_{<j}|=|H^X_{<6}|=|\{3,4,5\}|=3$, $\lfloor \frac{\mathbf{B}_o}{2}\rfloor=2$,  $S^X=\{1,2,6,7,8,9\}$, $H^X=\{2,3,4,5,6\}$, $S^{X'}=\{1,2,9\}$, $H^{X'}=\{2,3,4,5,6,7,8,9\}$ and $S^X_{<j}=S^{X'}_{<j}=\{1\}$.} \label{figure_large_r+1_2}
 \end{center}
\end{figure}
\begin{proposition}\label{nlargeOrt(T)=k}
Let $n\ge 2(r-k)+3\ge r+2$ and let $T\in \mathcal{T}$ be with $m$ blocks. Then the following holds.
\begin{itemize}
  \item[(i)] If $O(m)\ge k$, then $\Ort(T)=O(m)$;
  \item[(ii)] If $O(m)\le k-1$ then $\Ort(T)\le k-1$.
\end{itemize}
\end{proposition}
\begin{proof}Recall that  $O(m)=\lceil\frac{r+1-m}{2}\rceil$.

(i) As $k\le O(m)\le\frac{r+2-m}{2}$, then  $m\le r+2-2k$. Hence, as $\mathbf{B}_e\le m$ and $2(r-k)+3\le n$ we obtain that $r+1+\mathbf{B}_e\le r+1+m\le 2(r-k)+3\le n$, concluding that $r+1\le n-\mathbf{B}_e$. Thus, $\Ort(T)=O(m)$ by \Cref{lemmageneral:blocks-ortogonality} and the result holds in this case.

(ii)  Let $n=2(r-k)+3+l$, where $l\ge 0$ is an integer. As each even block have at least two elements, we have that $\mathbf{B}_e\le \lfloor\frac{n}{2}\rfloor$. If $n$ is odd, then $l$ is even and $\lfloor\frac{n}{2}\rfloor=r-k+1+\frac{l}{2}$. Similarly,   $\lfloor\frac{n}{2}\rfloor=r-k+1+\frac{l+1}{2}$ if $n$ is even. Hence, $\mathbf{B}_e\le \lfloor\frac{n}{2}\rfloor=r-k+1+\lceil\frac{l}{2}\rceil$ and so, $\mathbf{B}_e=r-k+1+\lceil\frac{l}{2}\rceil-j$ for some integer $j\ge 0$.
 On  the other hand, let $\alpha=r+1+\mathbf{B}_e-n$ and suppose first that  $\alpha\le 0$. Then $r+1\le n-\mathbf{B}_e$ and by \Cref{lemmageneral:blocks-ortogonality}, $\Ort(T)=O(m)\le k-1$, as desired. Now suppose that $\alpha>0$, then $\Ort(T)\le \frac{(n-\mathbf{B}_e)-m}{2}+\alpha+\lfloor\frac{\mathbf{B}_o}{2}\rfloor$ by  \Cref{lemmageneral:blocks-ortogonality2}. As $\alpha=r+1+(r-k+1+\lceil\frac{l}{2}\rceil-j)-(2(r-k)+3+l)=k-1-j-\lfloor\frac{l}{2}\rfloor$ and $m=\mathbf{B}_e+\mathbf{B}_o$, we obtain that
$\Ort(T)\le \frac{(2(r-k)+3+l)-(r-k+1+\lceil\frac{l}{2}\rceil)-(r-k+1+\lceil\frac{l}{2}\rceil+\mathbf{B}_o)}{2}+(k-1-j-\lfloor\frac{l}{2}\rfloor)+\lfloor\frac{\mathbf{B}_o}{2}\rfloor$ and thus,
$\Ort(T)\le \frac{1+\lfloor\frac{l}{2}\rfloor-\lceil\frac{l}{2}\rceil-\mathbf{B}_o}{2}+\frac{2k-2-2j-2\lfloor\frac{l}{2}\rfloor}{2}+\lfloor\frac{\mathbf{B}_o}{2}\rfloor$, concluding that
$$\Ort(T)\le \frac{2k-1-l-\mathbf{B}_o}{2}+\lfloor\frac{\mathbf{B}_o}{2}\rfloor.$$
If $\mathbf{B}_o$ is odd, then $\Ort(T)\le \frac{2k-1-l-\mathbf{B}_o}{2}+\frac{\mathbf{B}_o-1}{2}\le k-1$. If $\mathbf{B}_o$ is even, then $l\ge 1$ is odd. Therefore, $\Ort(T)\le \frac{2k-1-l-\mathbf{B}_o}{2}+\frac{\mathbf{B}_o}{2}\le \frac{2k-2-\mathbf{B}_o}{2}+\frac{\mathbf{B}_o}{2}=k-1$, concluding the proof.
\end{proof}

Recall that $\mathcal{O}_k(\M)=\{T\in \mathcal{T} \text{ } | \text{ } \Ort(T)=k\}$. For short, let us denote  $\mathcal{O}_k(\mathcal{C}_r(n))$ by $\mathcal{O}_k$. Let denote by $\mathcal{T}_{m}$ the set of $T\in \mathcal{T}$ with exactly $m$ blocks.
Next, we will obtain the $o$-vector of $\mathcal{C}_r(n)$, for $n$ large enough.

\begin{theorem}\label{O-formula_neighb}
If $n\ge 2(r-k)+1\ge r+2$, then
$$o(\mathcal{C}_r(n),i)=2\binom{n}{r-1-2i}$$ for every $i=k, \ldots, \lfloor\frac{r-1}{2}\rfloor$.
\end{theorem}
\begin{proof}
For every $k=0,\ldots, \lfloor\frac{r-1}{2}\rfloor$ let $k'=k+1$ and consider $m\in\{1,\dots,r\}$  such that  $\dfrac{r+2-m}{2}\in \mathbb{N}$ and $O(m)\geq k'$, where $O(m)=\lceil\frac{r+1-m}{2}\rceil$.

\smallskip

We first claim that $\mathcal{T}_{m}\cup  \mathcal{T}_{m-1} = \mathcal{O}_{O(m)}$.
Let $T \in \mathcal{T}_{m}\cup  \mathcal{T}_{m-1}$. By  \Cref{nlargeOrt(T)=k} (i), $\Ort(T)=O(m)$ since $O(m)\geq k'$ and $n\ge 2(r-k')+3$, concluding that $T\in \mathcal{O}_{O(m)}$. Now, let $T'$ be a tope of $\mathcal{C}_r(n)$ with $m'$ blocks and such that $T'\notin \mathcal{T}_{m}\cup  \mathcal{T}_{m-1}$. We will prove that $T'\notin \mathcal{O}_{O(m)}$. As  $O(m)=O(m-1)$ by the choice of $m$  and $m'\not\in\{m, m-1\}$, we obtain that $O(m')\neq O(m)$. If $O(m')\geq k'$, then $\Ort(T') = O(m')$  by  \Cref{nlargeOrt(T)=k} (i), concluding that $T\notin \mathcal{O}_{O(m)}$ since $\Ort(T') = O(m')\neq O(m)$. If $O(m')\le  k-1$, then $\Ort(T')\le k'-1$   by \Cref{nlargeOrt(T)=k} (ii). As  $O(m)\geq k'$, we obtain that  $\Ort(T')\neq O(m)$ and so, $T'\notin \mathcal{O}_{O(m)}$. Hence, $\mathcal{T}_{m}\cup  \mathcal{T}_{m-1} = \mathcal{O}_{O(m)}$ and the claim holds.
\smallskip

As $|\mathcal{T}_{m}|=2{n-1 \choose m-1}$ and $|\mathcal{T}_{m-1}|=2{n-1 \choose m-2}$ by \Cref{numberBlocks},  we obtain by the above claim that $|\mathcal{O}_{O(m)}|=2({n-1 \choose m-1} + {n-1 \choose m-2}) = 2{n \choose m-1}$.  On the other hand, as $O(m)=\frac{r+2-m}{2}$  by the choice of $m$, then  $m = r+2-2O(m)$ and so, $$|\mathcal{O}_{O(m)}| = 2{n \choose r+1-2O(m)}$$ for every $O(m)\geq k'$. 
Finally, as $o(\mathcal{C}_r(n),k)=|\mathcal{O}_{k+1}|$ for every $k=0,\ldots, \lfloor\frac{r-1}{2}\rfloor$ (\Cref{relation_ort_neighb}), the theorem holds.
\end{proof}

The above theorem is best possible. For instance,  \Cref{o_i(r+1)} shows that  $o(\mathcal{C}_r(n),k)\neq2$ if $n=r+1$ and $k=\frac{r-1}{2}$ (i.e., for $n<2(r-k)+1$), while the formula given in \Cref{O-formula_neighb} for these values give us $2\binom{n}{r-1-2k}=2$. Further,  \Cref{example_$o$-vector} shows  several values of $r,k$ and   $n<2(r-k)+1$, where $o(\mathcal{C}_r(n),k)\neq2\binom{n}{r-1-2k}$.


\smallskip

Below, we  obtain the $o$-vector of $\mathcal{C}_r(r+1)$.
 Given two sign-vectors $X, Y\in \{+,-\}^n$, we denote by ${X\cdot Y}$ the vector whose $i$-th entry is $-$ if $i \in S(X,T)$ and $+$ if $i \in H(X,T)$.

\begin{proposition}\label{o_i(r+1)}
Let $r\ge3$ and $0\le k\le\lfloor\frac{r-1}{2}\rfloor$, then
\begin{equation*}
     o(\mathcal{C}_r(r+1),k)=\left\{
	       \begin{array}{ll}
		 {r+1 \choose k+1}         & \text{ if    $k=\frac{r-1}{2}$; }  \\ \ \\
		 2 {r+1 \choose k+1}      & \text{ otherwise. }\\

	       \end{array}
	     \right.
\end{equation*}
\end{proposition}
\begin{proof}
 Let  $X$ be the unique circuit of $\mathcal{C}_r(r+1)$ that starts with the sign $+$. Consider  the set $\mathcal{Y}= \{Y\in \{+,-\}^{r+1}\}\setminus\{Z,-Z\}$, where  $Z\in \{+\}^{r+1}$. Now, let define the function $f:\mathcal{T}\rightarrow \mathcal{Y}$ as $f(T)=X\cdot T$. We will see that $f$ is  bijective. Consider any tope $T\in \mathcal{T}$ and notice that $X\cdot T\not\in\{Z,-Z\}$ since  by \Cref{B_rn=cyclicpoly}, $T$ has at most $r$ blocks. Thus, $f$ is  injective since for every $T,T'\in \mathcal{T}$, clearly $X\cdot T\neq X\cdot T'$ and $X\cdot T,X\cdot T'\in \mathcal{Y}$.
Now, let $T=(T_1,\ldots,T_{r+1})$ be such that
\begin{equation*}
    T_i = \left\{
	       \begin{array}{ll}
		 +     & \text{if  $i$ is odd and }Y_i=+ ; \\
		 -     & \text{if  $i$ is odd and } Y_i=- ;\\
		 +     & \text{if  $i$ is even and } Y_i=- ; \\
		 -     & \text{if  $i$ is even and } Y_i=+ ,\\
	       \end{array}
	     \right.
\end{equation*}
for every $i=1,\ldots,r+1$.
Then, we notice that $X\cdot T=Y$ since $X$ is alternating. Moreover, $T$ has at most $r$ blocks since  $Y\not\in\{Z,-Z\}$. Hence, $T\in \mathcal{T}$ by  \Cref{B_rn=cyclicpoly} and so,  $f$ is bijective.

Therefore,  as $\Ort(T) = X\perp T = \min\{(X\cdot T)^+, (X\cdot T)^-\}$,
 computing  $|\mathcal{O}_{k+1}(\mathcal{C}_r(r+1))|$ is equivalent to counting the number of sign-vectors $Y\in\mathcal{Y}$ with $|Y^+|=k+1$ and the number of sign-vectors $Y\in\mathcal{Y}$ with $|Y^-|=k+1$, for every $k=0,\ldots,\lfloor \frac{r-1}{2}\rfloor$, which is ${r+1 \choose k+1}$ if  $k=\frac{r-1}{2}$ and  $2 {r+1 \choose k+1}$  otherwise. Finally, as $|\mathcal{O}_{k+1}(\mathcal{C}_r(r+1))|=o(\mathcal{C}_r(r+1),k)=$ by \Cref{relation_ort_neighb}, the result holds.
\end{proof}


\subsection{Reducing $k$-Roudneff's conjecture to a finite case analysis and proving the case $k=\frac{r-1}{2}$}\label{section_general}


The main results of this subsection are that \Cref{quest:kRoudneff} can be reduced, for fixed $r$ and $k$, to uniform oriented matroids and a finite case analysis and that \Cref{quest:kRoudneff}  holds for $k=\frac{r-1}{2}$ (\Cref{prop:onlythebaseishard} and \Cref{r_odd_max_ort}).

\smallskip
Given a tope $T$ of $\mathcal{M}$ and  $e\in E$, denote by ${{_{-e}T}}$ the sign-vector obtained from reorienting the element $e$ in $T$.

\begin{lemma}\label{neighbors_max_ort} 
  Let $\mathcal{M}=(E,\mathcal{C})$ be an  oriented matroid of odd rank $r\ge3$ on $|E|\ge r+1$ elements and suppose that $\mathcal{M}$ has an $(\frac{r+1}{2})$-orthogonal tope $T$.  Then, $\Ort({{_{-e}T}})=\frac{r-1}{2}$  for every $e\in E$.
\end{lemma}
\begin{proof}
Let $e\in E$ and first consider a circuit $X\in \mathcal{C}$ with $e\not\in \underline{X}$ (if exists). Then, $X\perp {{_{-e}T}}=\frac{r+1}{2}$  since $|S(X,T)|=|S(X,{{_{-e}T}})|=|H(X,T)|=|H(X,{{_{-e}T}})|=\frac{r+1}{2}$. 
Now, consider $X\in \mathcal{C}$ such that $e\in \underline{X}$ 
and notice that $X\perp {{_{-e}T}}=\frac{r-1}{2}$.
Therefore, $X\perp {{_{-e}T}}\ge \frac{r-1}{2}$ for every circuit $X\in \mathcal{C}$, concluding that ${{_{-e}T}}$ is $(\frac{r-1}{2})$-orthogonal. Moreover, since there exists $X\in \mathcal{C}$ such that $X\perp T'=\frac{r-1}{2}$, we conclude that $\Ort({{_{-e}T}})=\frac{r-1}{2}$.
\end{proof}

For an oriented matroid $\mathcal{M}$ on $E$ elements, $e\in E$ and a tope $T$, recall that $T\setminus e$ is the sign-vector on ground set $E\setminus e$ and the same signs on these elements as $T$.




\begin{lemma}\label{k_ort_contract_deletion}
Let $\mathcal{M}$ be  a uniform  rank $r\ge3$
 oriented matroid  on $n=|E|$ elements and $e\in E$. If $T$ is a $k$-orthogonal tope of  $\mathcal{M}$,  $1\le k \le \lfloor\frac{r+1}{2}\rfloor$, then

 \begin{itemize}
     \item[(a)] $T\setminus e$ is a $k$-orthogonal tope in $\mathcal{M}\setminus e$;
     \item[(b)] $T\setminus e$ is a $k$-orthogonal tope in $\mathcal{M}/ e$ if ${{_{-e}T}}$ is $k$-orthogonal and $k<\frac{r+1}{2}$.
 \end{itemize}
\end{lemma}
\begin{proof}
 (a) By \Cref{prop:topegraph}, $B^{\mathcal{G}(\mathcal{M})}_{k-1}(T)\cong B^{Q_n}_{k-1}(T)$. The tope graph of $\mathcal{M}\setminus e$ can be obtained by contracting all the edges of $\mathcal{G}(\mathcal{M})$, corresponding to $e$. Hence,  $B^{\mathcal{G}(\mathcal{M}\setminus e)}_{k-1}(T\setminus e)\cong B^{Q_{n-1}}_{k-1}(T\setminus e)$ and so, $T\setminus e$ is a $k$-orthogonal tope in $\mathcal{M}\setminus e$.
 \smallskip

 (b) Notice that $\mathcal{M}/e$ has rank $r'=r-1$ and $k\le \lfloor\frac{r'+1}{2}\rfloor$.
If $T\setminus e$ is not a $k$-orthogonal tope in $\mathcal{M}/ e$, then there exists a circuit $Y$ of $\mathcal{M}/ e$ such that $(T\setminus e)\perp Y \le k-1$, but then
taking the circuit $\underline{X}=\underline{Y}\cup e$ of $\mathcal{M}$, we obtain that
$T\perp X\le k-1$ or ${{_{-e}T}}\perp X\le k-1$, since $T$ and ${{_{-e}T}}$ differ in exactly the entry $e$, contradicting then the fact that $T$ and ${{_{-e}T}}$ were $k$-orthogonal. Therefore,  $T\setminus e$ is a $k$-orthogonal tope and the result holds.
 \end{proof}
\Cref{figureB(3_5)B(3_4)} shows the tope graph of $\mathcal{M}=\mathcal{C}_3(5)$ and $\mathcal{M}\setminus e=\mathcal{C}_3(4)$, where $e=5$. Observe that $k$-orthogonal topes in $\mathcal{M}$ are mapped to $k$-orthogonal topes in $\mathcal{M}\setminus e$, for $k=1,2$.
Next, we will obtain a bound for $m(\M,k)$ in terms of $\mathcal{M}\setminus e$  and $\mathcal{M}/ e$.

\begin{lemma}\label{thm:onlythebaseishard}
 Let $\mathcal{M}$ be  a uniform  rank $r\ge3$
 oriented matroid  and let $e\in E$. Then,
$m(\M,k) \le m(\M\setminus e,k)+m(\M/ e,k)$
for every  $k=0,\ldots, \lfloor\frac{r-1}{2}\rfloor$.
\end{lemma}
\begin{proof}
Let $k'=k+1$ and let
 $T$ be a $k'$-orthogonal tope of $\mathcal{M}$. First suppose that  $k'=\frac{r+1}{2}$. Then by \Cref{neighbors_max_ort}, every tope adjacent to $T$ in $\mathcal{G}(\mathcal{M})$, in particular ${{_{-e}T}}$ is not $k'$-orthogonal. Hence, the mapping $\mathcal{M}^{\perp_{k'}}\to(\mathcal{M}\setminus e)^{\perp_{k'}}$ sending $T$ to $T\setminus e$ which by  \Cref{k_ort_contract_deletion} (a) is well-defined furthermore is injective. Hence, $|\mathcal{M}^{\perp_{k'}}|\le |(\mathcal{M}\setminus e)^{\perp_{k'}}|$ and   by  \Cref{relation_ort_neighb} we get $m(\M,k) \le m(\M\setminus e,k)$.

If $k'<\frac{r+1}{2}$
notice that $\mathcal{M}/e$ has rank $r'=r-1$ and so,  $k'\le \lfloor\frac{r'+1}{2}\rfloor$.
 If however, the neighbor ${{_{-e}T}}$ of $T$ with respect to $e$ was also $k'$-orthogonal, then both will be mapped to the same $k'$-orthogonal tope  of $\mathcal{M}\setminus e$, but then the tope $T\setminus e$ of $\mathcal{M}/e$  is $k'$-orthogonal by   \Cref{k_ort_contract_deletion} (b).
 Thus, we conclude that $|\mathcal{M}^{\perp_{k'}}|\leq|(\mathcal{M}\setminus e)^{\perp_{k'}}|+|(\mathcal{M}/e)^{\perp_{k'}}|$.
By  \Cref{relation_ort_neighb}, this yields $m(\M,k) \le m(\M\setminus e,k)+m(\M/ e,k)$.\end{proof}






An analogue of the above result in terms of $o(\mathcal{M},k)$ is not true since topes $T$ with $\Ort(T)=k$ in $\mathcal{M}$ do not necessarily satisfy that $\Ort(T\setminus e)=k$ in $\mathcal{M}\setminus e$ or $\mathcal{M}/ e$, even if its neighbor $T'$ has $\Ort(T')=k$ as in  \Cref{k_ort_contract_deletion}. For instance, \Cref{figureB(3_5)B(3_4)} shows examples of topes $T$ with $\Ort(T)=1$ and $\Ort(T\setminus e)=2$ in $\mathcal{M}\setminus e$. In fact, for $r=4$ and $n=6$ we have that $36=o(\mathcal{C}_4(6),0)>o(\mathcal{C}_4(6)\setminus e,0)+o(\mathcal{C}_4(6)/e,0)=o(\mathcal{C}_4(5),0)+o(\mathcal{C}_3(5),0)=10+20=30$, where these values are obtained from  \Cref{example_$o$-vector}.

\smallskip

The following reduces  \Cref{quest:kRoudneff} for fixed $r$ and $k$ to a finite number of cases.

\begin{theorem}\label{prop:onlythebaseishard} Let $r\ge 3$ and $0\le k\le \lfloor\frac{r-1}{2}\rfloor$.
If $m(\M',k)\leq c_{r'}(n',k)$ for every uniform oriented matroid $\M'$ of rank $r'\leq r$ on $n'=2(r'-k)+1$ elements, then $m(\M,k)\leq c_{r}(n,k)$ for every oriented matroid $\M$ of rank $r$ on $n\geq 2(r-k)+1$ elements.
\end{theorem}
\begin{proof}
Every rank $r$ oriented matroid $\mathcal{M}$ on $n$ elements can be perturbed to become a uniform rank $r$ oriented matroid $\mathcal{M}'$ on $n$ elements, see~\cite[Corollary 7.7.9]{BVSWZ99}. In particular, we have that the tope graph $\mathcal{G}(\mathcal{M})$ is a subgraph of the tope graph $\mathcal{G}(\mathcal{M}')$. Hence, $m(\mathcal{M},k)\le m(\mathcal{M}',k)$  
 for every $k=0,\ldots, \lfloor\frac{r-1}{2}\rfloor$.
So, let us consider a uniform rank $r$ oriented matroid $\mathcal{M}$ on $n$ elements  and let $e\in E$.




Let us show the claim by induction on $r$ and $n$. By \Cref{thm:onlythebaseishard},
$m(\mathcal{M},k)\le m(\M\setminus e,k)+m(\M/e,k)$. Now, fix $r$ and let  $n>2(r-k)+1$. Notice that the inequality  $m(\M\setminus e,k)\le c_r(n-1,k)$ then follows by induction on $n$ since we know that it is verified for all uniform rank $r$ oriented matroids  on $n-1$ elements.
On the other hand, the inequality $m(\M/e,k)\le  c_{r-1}(n-1,k)$ follows since by assumption all uniform rank $r'=r-1$  oriented matroid $\mathcal{M}'$ on $n'=2(r'-k)+1$ elements satisfy $m(\M',k)\le c_{r}(n',k)$. Thus by induction this also holds for $n-1\ge 2(r-k)+1\ge 2(r'-k)+1$.
Now, a straight-forward computation using \Cref{O-formula_neighb} and the fact that $\sum\limits_{i=k}^{\lfloor\frac{r-1}{2}\rfloor} o(\mathcal{C}_r(n),i)=c_r(n,k)$, yields
$$c_{r}(n-1,k)+c_{r-1}(n-1,k)=c_{r}(n,k).$$ Thus, we obtain that $m(\mathcal{M},k)\le c_r(n,k)$.
\end{proof}

The following result answers \Cref{quest:kRoudneff} in the affirmative  for  odd $r$ and $k=\frac{r+1}{2}$.


\begin{corollary}\label{r_odd_max_ort}
  Let $\mathcal{M}$ be an oriented matroid of odd rank $r\ge3$ on $n\ge r+2$ elements and $k=\frac{r-1}{2}$. Then, $m(\mathcal{M},k)\le c_r(n,k)=2$.
  \end{corollary}

 \begin{proof}

   As $k=\frac{r-1}{2}$, we notice that $o(\mathcal{M},k)=m(\mathcal{M},k)$ and $o(\mathcal{C}_r(n),k)=c_r(n,k)$. Moreover, $c_r(n,k)=2$ by \Cref{O-formula_neighb}.
  In order to prove that $m(\mathcal{M},k)\le c_r(n,k)$ for $n\ge r+2$, it is sufficient by \Cref{prop:onlythebaseishard} to verify it for all uniform rank $r$ oriented matroid $\mathcal{M}$ on $n=r+2=2(r-k)+1$ elements, since for smaller rank $r'<r$ we will have $k>\frac{r'-1}{2}$ and hence the $k$-entries of all $o$-vectors are $0$. As in that case there is only one reorientation class (see \Cref{oneclass_r+2}), we obtain that
$o(\mathcal{M},k)=o(\mathcal{C}_r(n),k)$,  concluding the proof.
\end{proof}



\section{$k$-McMullen's problem and $k$-Roudneff's conjecture for low ranks}\label{computer}


\subsection{A computer program that obtains $o(\M,k)$}\label{computer_program}

Next, we explain  how we can obtain the $o$-vector of a uniform oriented matroid from its chirotope:

\medskip

\noindent\emph{chirotope $\rightarrow$ circuits:}

\noindent In a uniform oriented matroid the chirotope $\chi(B)\neq 0$ for every ordered set of size $r$. Further, the supports of its circuits correspond to all sets of size $r+1$.
It is well-known, see \cite[Section 3.5]{BVSWZ99}, that from the chirotope of a uniform oriented matroid, we obtain the signs of a circuit $X$ with $\underline{X}=\{b_1,...,b_{r+1}\}$ via:
$$\chi(B)=-X_{b_i}\cdot X_{b_{i+1}}\cdot \chi(B'),$$
where $B=\underline{X}\setminus b_i$ and $B'=\underline{X}\setminus b_{i+1}$.
This allows us to compute the set $\mathcal{C}$ from $\chi$.

\medskip

\noindent\emph{circuits $\rightarrow$ $o$-vector:}

\noindent For any sign-vector $T\in \{+,-\}^n$, we obtain $\Ort(T)=\min\{X\perp T \text{ } | \text{ } X\in \mathcal{C}\}$ and so, using the correspondence $o(\M,k)=|\mathcal{O}_{k+1}(\M)|$ given in \Cref{relation_ort_neighb}, the $o$-vector of $\mathcal{M}$.

\bigskip

Finschi and Fukuda \cite{FinschiFukuda2002,Finschi2001}
generated (up to isomorphism) all the chirotopes of uniform rank~$r$ oriented matroids on~$n$ elements,  for $4\le r\le7$ and $n=r+3$, and moreover classified them by realizability and for $r=5$ and $n=9$ (where
some of the data and also their source code for the enumeration is available only upon request from Lukas Finschi).
We implemented the above procedure in  Python (available at
\cite{supplemental_data})
giving us the $o$-vector of all the reorientation classes from the database. 
We resume the results in  the following theorem:

\begin{theorem}\label{thm_computer}
 Let $\mathcal{M}\not\in [C_r(n)]$ be a  uniform rank $r$ oriented matroid on $n$ elements, then the following hold:
 \begin{itemize}
    \item [(a)]  if $r=5$ and $n=8$, then $o(\mathcal{M},1)<o(\mathcal{C}_5(8),1)$,  $m(\mathcal{M},1)<c_5(8,1)$, $m(\mathcal{M},2)\le c_5(8,2)$ and there are exactly $3$ reorientation  classes with $m(\mathcal{M},2)=c_5(8,2)$. 
   Moreover,
     there exists $\mathcal{M}$ realizable such that $m(\mathcal{M},2)=0$; 
    \item [(b)] if $r=5$ and $n=9$, then $o(\mathcal{M},1)<o(\mathcal{C}_5(9),1)$, $m(\mathcal{M},1)<c_5(9,1)$, $m(\mathcal{M},2)\le c_5(9,2)$ and there are exactly $23$ reorientation classes with $m(\mathcal{M},2)=c_5(9,2)$; 
    \item [(c)] if $r=6$ and $n=9$, then $m(\mathcal{M},1)<c_6(9,1)$ and $m(\mathcal{M},2)<c_6(9,2)$. Moreover, there are exactly $91$ reorientation classes having $o(\mathcal{M},1)>o(\mathcal{C}_6(9),1)$ and there exists
    $\mathcal{M}$ realizable such that $m(\mathcal{M},2)=0$; 
    \item [(d)] 
     if $r=7$ and $n=10$, then $m(\mathcal{M},1)<c_7(10,1)$,
 $0<o(\mathcal{M},2)<o(\mathcal{C}_7(10),2)$, $m(\mathcal{M},2)<c_7(10,2)$,  $m(\mathcal{M},3)\le c_7(10,3)$ and there are exactly $37$ reorientation classes  with  $m(\mathcal{M},3)=c_7(10,3)$. Moreover, there are exactly $312336$ reorientation classes having $o(\mathcal{M},1)>o(\mathcal{C}_7(10),1)$ and
 there exists  $\mathcal{M}$ realizable such that $m(\mathcal{M},3)=0$. 
 \end{itemize}
\end{theorem}
\begin{example}\label{example_$o$-vector}
  Since the chirotope of the alternating oriented matroid is always $+$, using our computer program we compute the $o$-vector of $\mathcal{C}_r(n)$ for some values of $r$ and $n$  in Figure \ref{table}. For all these values and $k$ such that $n\le2(r-k)$, we notice that $o(\mathcal{C}_r(n),k)\neq2\binom{n}{r-1-2k}$, showing that  \Cref{O-formula_neighb} is best possible.
\end{example}

 \begin{figure}[htb]
\begin{center}
 \includegraphics[width=1.05\textwidth]{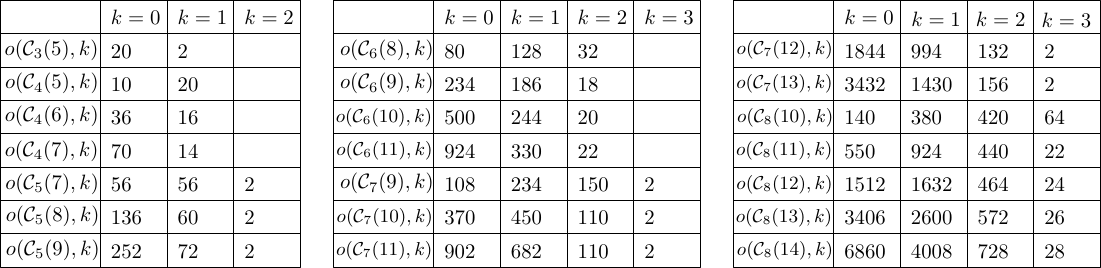}
 \caption{} \label{table}
 \end{center}
\end{figure}
\subsection{Results on $k$-McMullen's problem}
Next, we answer \Cref{question_k-McMullen} affirmatively for  $(r,k)\in (5,2),(6,2),(7,3)$. Further, we show that the lower bound in \Cref{question_k-McMullen} is tight in one more case, i.e., $10\le \nu(7,2)$.

\begin{theorem}\label{small_values_kMcMullen}
We have
   $\nu(5,2)=\nu_{\mathrm{R}}(5,2)=7$, $\nu(6,2)=\nu_{\mathrm{R}}(6,2)=8$, $\nu(7,3)=\nu_{\mathrm{R}}(7,3)=9$, and
$10\le \nu_{\mathrm{R}}(7,2)$. 
\end{theorem}
\begin{proof}
By \Cref{McMullenr+2}, we have $7\le\nu(5,2)$, $8\le\nu(6,2)$
and $9\le\nu(7,3)$ (and so, also $7\le\nu_{\mathrm{R}}(5,2)$, $8\le\nu_{\mathrm{R}}(6,2)$
and $9\le\nu_{\mathrm{R}}(7,3)$). The lower bound  $10\le \nu(7,2)$ holds since  by \Cref{thm_computer} (d), $0<o(\mathcal{M},2)$ 
for any uniform rank $7$ oriented matroid  $\mathcal{M}$ on $10$ elements.
 On the other hand, the upper bounds $\nu_{\mathrm{R}}(5,2)<8$, $\nu_{\mathrm{R}}(6,2)<9$ and $\nu_{\mathrm{R}}(7,3)<10$ hold by \Cref{thm_computer} (a), (c) and (d), respectively, since there exists a rank $r$ realizable uniform oriented matroid $\mathcal{M}$ on $n$ elements such that $m(\mathcal{M},k)=0$, for $(r,k,n)\in (5,2,8),(6,2,9),(7,3,10)$.  
 Then,  $\nu(5,2)<8$, $\nu(6,2)<9$ and $\nu(7,3)<10$ and the result follows.
\end{proof}

\subsection{Results on $k$-Roudneff's conjecture}
 \Cref{quest:kRoudneff} can be reduced to uniform oriented matroids and reduce to a finite problem for fixed $r$ and $k$ by \Cref{prop:onlythebaseishard}. Next, we may
 answer \Cref{quest:kRoudneff} affirmatively for $r=6$ and $k=2$. 

\begin{theorem}\label{coro_cyclique_unique} 
    Let $\mathcal{M}$ be a rank $6$ oriented matroid on $n\ge 9$ elements. Then, $m(\mathcal{M},2)\le c_6(n,2).$
\end{theorem}
\begin{proof}
By \Cref{thm_computer} (c),  $m(\mathcal{M},2)\le c_6(9,2)$
 for all rank $6$  uniform  oriented matroid  $\mathcal{M}$ on $9$ elements.  On the other hand, it is known that
 $m(\mathcal{M}',2)\le c_5(n,2)$
 for any rank $5$ oriented matroid $\mathcal{M}'$   on $n\ge 7$ elements by \Cref{r_odd_max_ort}. 
 Then, the result follows for $n\ge 9$ by  \Cref{prop:onlythebaseishard}.
\end{proof}

\subsubsection*{Acknowledgements.} We thank Helena Bergold and Manfred Scheucher for fruitful discussions and also M. Scheucher for being the intermediary who provided us with the data and the source code of L. Finschi.
 R. Hern\'andez-Ortiz and L.\ P.\ Montejano were supported by SGR Grant 2021 00115.
K. Knauer was supported by the Spanish State Research Agency
through grants RYC-2017-22701, PID2019-104844GB-I00, PID2022-137283NB-C22 and the Severo Ochoa and María de Maeztu Program for Centers and Units of Excellence in R\&D (CEX2020-001084-M).

\bibliographystyle{my-siam}
\bibliography{lit}

\newpage
\section*{appendix}\label{appendix}\label{appendix}
We present the formulation of McMullen's problem and Roudneff's conjecture in their original versions. This is done merely out of illustrative reasons and none of our proofs or arguments outside of this appendix are based on what we will briefly explain below. 

\subsection*{McMullen's problem} 

Oriented matroids generalize the oriented linear algebra of Euclidean space in the following way: let $M\in\mathbb{E}^{m\times n}$ a real $m\times n$ matrix, then $\M_M=([n],\mathcal{C}_M)$ is an oriented matroid where $X\in \mathcal{C}_M$ if there is a minimal linear combination $\lambda_{i_1}c_{i_1}+\ldots+\lambda_{i_k}c_{i_k}=\mathbf{0}$ of columns of $M$ such that $\underline{X}=\{i_1, \ldots, i_k\}$ and $X_{i_j}$ is the sign of $\lambda_{i_j}$ for all $1\leq j\leq k$. The rank of $\M_M$ is the rank of $M$. If $\M$ arises this way from a matrix $M$, it is called \emph{realizable}. Realizable oriented matroids form a small subclass of all oriented matroids~\cite[Corollary 7.4.3]{BVSWZ99}, but capture hyperplane arrangements, point configurations, linear programming, and directed graphs.

\smallskip

McMullen's problem (\Cref{conjecture_McMullen})  was originally formulated for realizable uniform oriented matroids  in the language of projective transformations~\cite{L72} and then, more generally in terms of uniform oriented matroids \cite{CS85}. A \emph{projective transformation} $T:\mathbb{R}^{d} \rightarrow \mathbb{R}^{d}$ is a function such that $T(x)=\frac{Ax+b}{\langle c,x \rangle + \delta}$, where $A$ is a linear transformation of $\mathbb{R}^{d}$, $b, c \in \mathbb{R}^{d}$ and $\delta \in \mathbb{R},$ is such that at least one of $c\neq 0$ or $\delta \neq 0$. Further, $T$ is said to be \emph{permissible} for a set $X\subset \mathbb{R}^{d}$ if  $\langle c,x \rangle + \delta \neq 0$ for all $x\in X$ (see \cite[Appendix 2.6]{ZIEG1995}). It turns out that permissible projective transformations on $n$ points in $\R^d$ correspond to acyclic reorientations of the corresponding rank $r=d+1$ oriented matroid $\mathcal{M}$ on $n$ elements, see \cite[Theorem 1.2]{CS85}.
Therefore, as a set $X$ of $n$ points in general position in $\mathbb{R}^{d}$ corresponds to a realizable uniform oriented matroid $\M$ of rank $r=d+1$ on $n$ elements and  a permissible projective transformation leading $X$ to the vertices of a convex polytope corresponds to a $1$-neighborly reorientation of $\M$, we get Conjecture \ref{conjecture_McMullen} in its original version \cite[P.1.]{L72}.

\begin{problem*}[McMullen 1972]
Determine the largest number $\nu(d)$ such that any set of $\nu(d)$ points, lying in general position in $\mathbb{R}^{d}$, may be mapped by a permissible projective transformation onto the set of vertices of a convex polytope.
\end{problem*}

\subsection*{Roudneff's conjecture} 

The Topological Representation Theorem states that the reorientation classes of simple oriented matroids on $n$ elements and rank $r$ are in one-to-one correspondence with the classes of isomorphism of arrangements of $n$ pseudospheres in~$S^{r-1}$ \cite{FL78}. See  \cite[Theorem 1.4.1]{BVSWZ99} for the precise definitions of such arrangements. There is a natural identification between pseudospheres and pseudohyperplanes as follows: Recall that $\p^{r-1}$ is the topological space obtained from $S^{r-1}$ by identifying all pairs of antipodal points. The double covering map $\pi: S^{r-1} \rightarrow \p^{r-1}$, given by $\pi(x)=\{x,-x\}$, gives an identification of centrally symmetric subsets of $S^{r-1}$ and general subsets of $\p^{r-1}$. This way centrally symmetric pseudospheres in $S^{r-1}$ correspond to pseudohyperplanes in $\p^{r-1}$. Since the pseudoshperes in the Topological Representation Theorem can be assumed to be centrally symmetric, we get a  statement in terms of pseudohyperplanes in $\p^{r-1}$, i.e., the reorientation classes of simple oriented matroids on $n$ elements and rank $r$ are in one-to-one correspondence with the classes of isomorphism of arrangements of $n$ pseudohyperplanes in~$\p^{r-1}$. See~\cite[Exercise 5.8]{BVSWZ99}. In this model one usually uses the \emph{dimension} $d=r-1$ of the rank. Given an arrangement $\mathcal{H}(d,n)$ of $n$ pseudohyperplanes in $\p^{d}$ representing an reorientation class of $[\M]$, any given element of a class can be obtained by choosing for each pseudohyperplane $H_e\in\mathcal{H}(d,n)$ in which of its two sides is positive and which is negative. Now, every point in $x\in \p^{d}$ yields a sign-vector $X$, where $X_e$ is $0,+,-$ depending on whether $x$ lies on $H_e$, on its positive side, or its negative side.


\smallskip

An arrangement $\mathcal{H}(d,n)$ of $n$ pseudohyperplanes in $\p^{d}$ is called {\em simple} if every intersection of $d$ pseudohyperplanes is a unique distinct point. Simple arrangements correspond to reorientation classes of uniform oriented matroids. The maximal cells of the arrangement $\mathcal{H}(d,n)$ correspond to one half of the topes of the oriented matroid $\mathcal{M}$ (obtained by factoring the antipodal map). The topes then corresponds to the acyclic reorientation of $\mathcal{M}$, by orienting the pseudohyperplanes such that points inside the corresponding maximal cell are all-positive.
A complete cell of $\mathcal{H}(d,n)$ is a maximal cell that is bounded by every hyperplane of the arrangement. In the corresponding reorientation  of $\M$, reorienting any element of $E$ results in another acyclic reorientation, i.e., the one corresponding to the adjacent maximal cell. This is, complete cells correspond to $1$-neighborly reorientations.
\smallskip

On the other hand, \emph{Cyclic
arrangements} are defined as the lattice theoretical dual or polar of the point set given by the vertices of the cyclic polytope  which hence have $\frac{1}{2}c_{r-1}(n,1)$ complete cells.
Since the latter number for $n\geq 2r-1$ can be expressed as a sum of binomial coefficients (see~\cite[Theorem 2.1]{R99}), we get Conjecture \ref{conj1} in its original version\cite[Conjecture 2.2]{R99}.
\begin{conjecture*}[Roudneff 1991]
    Every arrangement of $n\geq 2d+1\geq 5$ (pseudo)hyperplanes in $\p^d$ has at most $\sum_{i=0}^{d-2}{n-1 \choose i}$ complete cells.
\end{conjecture*}




\end{document}